%% file: cvi_operator.tex
\documentclass{amsart}
\input{header}
\usepackage{mathtools}

\begin{document}

\title[Conformally covariant polydifferential operators]{Some constructions of formally self-adjoint conformally covariant polydifferential operators}
\author{Jeffrey S. Case}
\address{109 McAllister Building \\ Penn State University \\ University Park, PA 16802 \\ USA}
\email{jscase@psu.edu}
 \author{Yueh-Ju Lin}
 \address{Department of Mathematics, Statistics, and Physics \\ Wichita State University \\ Wichita, KS 67260 \\ USA}
 \email{lin@math.wichita.edu}
 \author{Wei Yuan}
 \address{Department of Mathematics \\ Sun Yat-sen University \\ Guangzhou, Guangdong
 510275 \\ China}
 \email{yuanw9@mail.sysu.edu.cn}
\begin{abstract}
 We introduce the notion of formally self-adjoint conformally covariant polydifferential operators and give some constructions of families of such operators.  In one direction, we show that any homogeneous conformally variational scalar Riemannian invariant (CVI) induces one of these operators.  In another direction, we use the ambient metric to give alternative constructions of certain operators produced this way, one of which is a formally self-adjoint, fourth-order, conformally covariant tridifferential operator which should be regarded as the simplest fully nonlinear analogue of the Paneitz operator.
\end{abstract}
\maketitle

\input{intro}
\input{bg}
\input{defn}
\input{operator}
\input{weight}
\input{ambient}
\input{trilinear}
\input{ovsienko_redou}
\input{comments}

\subsection*{Acknowledgments}
JSC was supported by the Simons Foundation (Grant \# 524601).  WY was supported by NSFC (Grant No.\ 12071489, No.\ 12025109).  Part of this work was carried out while YL was employed at Princeton University.

\bibliographystyle{abbrv}
\bibliography{bib}
\end{document}

%% file: header.tex
\usepackage{amsmath}
\usepackage{amssymb}
\usepackage{amsthm}
\usepackage{mathrsfs}

\DeclareMathOperator{\Id}{Id}

\DeclareMathOperator{\im}{im}
\DeclareMathOperator{\tr}{tr}

\DeclareMathOperator{\dvol}{dvol}

\DeclareMathOperator{\supp}{supp}

\DeclareMathOperator{\Ric}{Ric}

\DeclareMathOperator{\End}{End}

\DeclareMathOperator{\Sym}{Sym}

\newcommand{\cd}{\widetilde{d}}
\newcommand{\cg}{\widetilde{g}}

\newcommand{\cf}{\widetilde{f}}

\newcommand{\cu}{\widetilde{u}}
\newcommand{\cv}{\widetilde{v}}
\newcommand{\cw}{\widetilde{w}}

\newcommand{\cz}{\widetilde{z}}
\newcommand{\cB}{\widetilde{B}}
\newcommand{\cC}{\widetilde{C}}
\newcommand{\cD}{\widetilde{D}}

\newcommand{\cF}{\widetilde{F}}
\newcommand{\cG}{\widetilde{G}}
\newcommand{\cH}{\widetilde{H}}
\newcommand{\cK}{\widetilde{K}}
\newcommand{\cL}{\widetilde{L}}

\newcommand{\cphi}{\widetilde{\phi}}

\newcommand{\cnabla}{\widetilde{\nabla}}
\newcommand{\cdelta}{\widetilde{\delta}}
\newcommand{\cDelta}{\widetilde{\Delta}}
\newcommand{\cmE}{\widetilde{\mathcal{E}}}

\newcommand{\cmG}{\widetilde{\mathcal{G}}}

\newcommand{\hg}{\widehat{g}}

\newcommand{\lp}{\langle}
\newcommand{\rp}{\rangle}
\newcommand{\lv}{\lvert}
\newcommand{\rv}{\rvert}

\newcommand{\mE}{\mathcal{E}}

\newcommand{\mG}{\mathcal{G}}

\newcommand{\mR}{\mathcal{R}}
\newcommand{\mS}{\mathcal{S}}

\newcommand{\kC}{\mathfrak{C}}
\newcommand{\kD}{\mathfrak{D}}

\newcommand{\kM}{\mathfrak{M}}

\newcommand{\bN}{\mathbb{N}}

\newcommand{\bR}{\mathbb{R}}

\newcommand{\Lin}{\mathscr{L}}
\newcommand{\Poly}{\mathscr{P}}
\newcommand{\FSA}{\mathscr{F}}

\def\sideremark#1{\ifvmode\leavevmode\fi\vadjust{\vbox to0pt{\vss
 \hbox to 0pt{\hskip\hsize\hskip1em
 \vbox{\hsize3cm\tiny\raggedright\pretolerance10000
 \noindent #1\hfill}\hss}\vbox to8pt{\vfil}\vss}}}

\newcommand{\suchthat}{\mathrel{}\middle|\mathrel{}}

\newcommand{\comment}[1]{}

\newtheorem{thm}{Theorem}[section]
\newtheorem{prop}[thm]{Proposition}
\newtheorem{lem}[thm]{Lemma}
\newtheorem{cor}[thm]{Corollary}

\theoremstyle{definition}
\newtheorem{defn}[thm]{Definition}

\newtheorem{conj}[thm]{Conjecture}

\theoremstyle{remark}
\newtheorem{remark}[thm]{Remark}

\numberwithin{equation}{section}

%% file: intro.tex
\section{Introduction}
\label{sec:intro}

The GJMS operators $L_{2k}$~\cite{GJMS1992} are a family of conformally covariant differential operators with leading-order term $(-\Delta)^k$ defined on Riemannian manifolds of dimension $n\geq2k$.  The conformal Laplacian $L_2$ and the Paneitz operator $L_4$ are special cases of particular interest.  Due to their conformal invariance, these operators appear in many related contexts, including the Yamabe Problem (e.g.\ \cite{LeeParker1987}), the problem of prescribing Branson's $Q$-curvatures (e.g.\ \cite{ChangYang1995,GurskyHangLin2016}), the determination of the sharp constant in the Sobolev inequality controlling the continuous embedding $W^{k,2}(S^n)\subset L^{\frac{2n}{n-2k}}(S^n)$, $n>2k$ (e.g.\ \cite{Beckner1993,Branson1995}), and as intertwining operators in the spherical principal series representations for $SO_0(n+1,1)$ (e.g.\ \cite{Branson1995}).

Branson's $Q$-curvatures~\cite{Branson1995} are, roughly speaking, the zeroth-order terms of the GJMS operators.  Another interesting family of scalar Riemannian invariants are the renormalized volume coefficients $v_k$~\cite{Graham2000}.  When $k\leq2$ or the underlying manifold is locally conformally flat, these reduce to the $\sigma_k$-curvatures~\cite{Graham2009,Viaclovsky2000}.  The interest in the renormalized volume coefficients stems from the fact that they are variational within conformal classes and well-behaved at Einstein metrics~\cite{ChangFang2008,ChangFangGraham2012,Graham2009}, give rise to a family of sharp fully nonlinear Sobolev inequalities~\cite{GuanWang2004}, and, at least when they reduce to the $\sigma_k$-curvatures, the associated Yamabe-type problem can be solved~\cite{GuanWang2004,ShengTrudingerWang2007}.  For our purposes, the primary difference between the $Q$-curvatures and the renormalized volume coefficients is that their conformal transformations are controlled by linear and fully nonlinear operators, respectively.

In our previous work~\cite{CaseLinYuan2016}, we showed that many similarities between the approaches to prescribing the $Q$-curvatures and the $\sigma_k$-curvatures, such as local stability and rigidity results, are general features of conformally variational invariants (CVIs); see also~\cite{GoverOrsted2010} for a general treatment of Kazdan--Warner obstructions.  A \emph{CVI} is a natural Riemannian scalar invariant $L$ --- meaning it is given by a formula which is a linear combination of complete contractions of tensor products of the Riemannian metric, its inverse, the Riemann curvature tensor, and its covariant derivatives --- which is homogeneous in the metric and variational within conformal classes\footnote{These are really the \emph{even} invariants.  We restrict our attention to the even invariants for simplicity, but expect our arguments extend to all invariants.}.  Homogeneity means that there is a $w\in\bR$, called the \emph{weight}, such that $L^{c^2g}=c^wL^g$ for all metrics $g$ and all positive constants $c$; if $L$ is not identically zero, then $w\in-2\bN_0$.  A natural Riemannian scalar invariant $L$ is \emph{variational} within conformal classes if there is a Riemannian functional $\mS$, in the sense of Definition~\ref{defn:cvi} below, such that
\[ \left.\frac{d}{dt}\right|_{t=0}\mS\left(e^{2t\Upsilon}g\right) = \int_M \Upsilon L^g\,\dvol_g \]
for all Riemannian manifolds $(M,g)$ and all $\Upsilon\in C_0^\infty(M)$.  Equivalently, $L$ is variational if the conformal linearization $\left.\frac{\partial}{\partial t}\right|_{t=0}L^{e^{2t\Upsilon}g}$ is formally self-adjoint for all metrics $g$~\cite{BransonGover2008}.

The primary purpose of this article is to show that given any CVI $L$, there is a formally self-adjoint conformally covariant polydifferential operator $D$ which is naturally associated to $L$, in the sense that the formula for $L^{e^{2\Upsilon}g}$ can be written entirely in terms of $D$ and related objects.  This relation is captured precisely in~\eqref{eqn:pre_branson_trick} and~\eqref{eqn:branson_trick} below, and a precise existence statement is given in Theorem~\ref{thm:existence}.  Here a \emph{polydifferential operator} is a multilinear operator on smooth functions with the property that fixing all but one input yields a differential operator of finite order.  In fact, if $L$ has weight $-2k$, then this restriction will have order at most $2k$.

This work is motivated by the recent discovery of some useful formally self-adjoint conformally covariant polydifferential operators:

In conformal geometry, Case and Wang~\cite{CaseWang2016s} constructed such operators associated to the $\sigma_k$-curvatures in the variational cases and used them to study the Dirichlet problem on manifolds with boundary.  For example, the formally self-adjoint conformally covariant tridifferential operator associated to the $\sigma_2$-curvature is the polarization of
\begin{multline}
 \label{eqn:sigma2_operator}
 L_4(u) := \frac{1}{2}\delta\left(\lv\nabla u\rv^2\,du\right) - \frac{n-4}{16}\left( u\Delta\lv\nabla u\rv^2 - \delta\left((\Delta u^2)\,du\right) \right) \\
  - \frac{1}{2}\left(\frac{n-4}{4}\right)^2u\delta\left(T_1(\nabla u^2)\right) + \left(\frac{n-4}{4}\right)^2\sigma_2u^3 ;
\end{multline}
see~\cite[Remark~2.2]{Case2019fl}.  By exploiting the multilinearity of these operators, Case and Wang solved a Dirichlet problem on manifolds with boundary under a natural positivity assumption~\cite{CaseWang2016s} and made partial progress on the conjectured fully nonlinear sharp Sobolev trace inequality related to the $\sigma_2$-curvature~\cite{CaseWang2019}.  Case also used the multilinearity of $L_2$ to give a new proof~\cite{Case2019fl} of a sharp fully nonlinear Sobolev inequality involving $\sigma_2$-curvature (cf.\ \cite{ChangGurskyYang2003b,GuanWang2004,LiLi2003,Viaclovsky2000}).

In CR geometry, the $Q^\prime$-curvature~\cite{CaseYang2012,Hirachi2013} is a local pseudohermitian invariant of order $2n$ defined on pseudo-Einstein $(2n-1)$-dimensional manifolds which transforms quadratically under change of contact form:
\begin{equation}
 \label{eqn:Qprime}
 e^{n\Upsilon}Q_{e^\Upsilon\theta}^\prime = Q_\theta^\prime + P_\theta^\prime(\Upsilon) + \frac{1}{2}P_\theta(\Upsilon^2) .
\end{equation}
Its importance stems from the fact that the total $Q^\prime$-curvature has properties analogous to the total $Q$-curvature of an even-dimensional Riemannian manifold.  In the terminology to be introduced below, the polarization of $\Upsilon\mapsto P_\theta(\Upsilon^2)$ is the CR covariant bidifferential operator associated to $Q^\prime$.  The limiting case of~\eqref{eqn:branson_trick} below shows that, in the critical dimension, a formula similar to~\eqref{eqn:Qprime} holds between all CVIs and the conformally covariant operators associated to them.

There are some other constructions of conformally covariant polydifferential operators in the literature, especially in the conformally flat category.  In dimension two, one has the well-known (tridifferential) Schwarzian derivative and the (bidifferential) Rankin--Cohen bracket~\cite{Cohen1975,Rankin1956}.  More generally, for all but finitely many possible choices of weights, Ovsienko and Redou~\cite{OvsienkoRedou2003} classified conformally covariant bidifferential operators on round spheres of any dimension.  The restriction on the weights was removed by Clerc~\cite{Clerc2016,Clerc2017} as part of his detailed study of conformally covariant trilinear forms.  We are not aware of constructions of conformally covariant polydifferential operators of higher rank (i.e.\ which are not linear) on general manifolds in the literature.  As discussed further below, we show that the subfamily of formally self-adjoint operators in the Ovsienko--Redou classification exist on all Riemannian manifolds of sufficiently large dimension.

We now state our results more explicitly.  In order to have more economical terminology, we introduce a few definitions. 

\begin{defn}
 \label{defn:natural}
 Let $\ell\in\bN_0$.  A \emph{natural $\ell$-differential operator} is an operator $D\colon \bigl(C^\infty(M)\bigr)^\ell\to C^\infty(M)$ defined on any Riemannian manifold $(M^n,g)$ of sufficiently large dimension such that $D(u_1,\dotsc,u_\ell)$ can be expressed as a linear combination of complete contractions of tensor products of covariant derivatives of the functions $u_j$, $1\leq j\leq\ell$, the Riemannian metric, its inverse, the Riemann curvature tensor and its covariant derivatives.
 
 A \emph{natural polydifferential operator} is an operator which is a natural $\ell$-differential operator for some $\ell\in\bN$.
\end{defn}

Note that a natural $0$-differential operator is a natural scalar Riemannian invariant.  Fixing all but one of the inputs of a natural polydifferential operator yields a (linear) differential operator.

In what follows, given a natural polydifferential operator $D$, we shall always assume that we restrict to Riemannian manifolds of a dimension on which $D$ is defined.  In practice, we will study natural polydifferential operators which are homogeneous of degree $-2k$ in $g$ and defined on all Riemannian manifolds of dimension $n\geq2k$.

\begin{defn}
 \label{defn:conformally_covariant}
 A natural $\ell$-differential operator $D$ is \emph{conformally covariant} if for any $n\in\bN$, there are constants $a,b\in\bR$ such that
 \begin{equation}
  \label{eqn:conformally_covariant}
  D^{e^{2\Upsilon}g}(u_1,\dotsc,u_\ell) = e^{-b\Upsilon} D^g(e^{a\Upsilon}u_1,\dotsc,e^{a\Upsilon}u_\ell)
 \end{equation}
 for all Riemannian manifolds $(M^n,g)$ and all $\Upsilon,u_1,\dotsc,u_\ell\in C^\infty(M)$.  We call $(a,b)$ the \emph{bidegree} of $D$.
\end{defn}

In particular, conformally covariant operators are homogeneous in the metric.

\begin{defn}
 \label{defn:formally_self-adjoint}
 A natural $\ell$-differential operator $D$ is \emph{formally self-adjoint} if for every Riemannian manifold $(M^n,g)$ and every $u_0,\dotsc,u_\ell\in C^\infty(M)$ such that $u_0\dotsm u_\ell$ has compact support, the map
 \[ (u_0,\dotsc,u_\ell) \mapsto \int_M u_0\,D(u_1,\dotsc,u_\ell)\,\dvol_g \]
 is symmetric.
\end{defn}

Note that if $D$ is formally self-adjoint, then it is necessarily symmetric; i.e.\ the map $(u_1, \dotsc, u_\ell) \mapsto D(u_1, \dotsc, u_\ell)$ is symmetric.

The bidegree of a formally self-adjoint conformally covariant operator is determined entirely by its degree of homogeneity; see Lemma~\ref{lem:find_weights}.

\begin{defn}
 \label{defn:recovers}
 Let $L$ be a CVI of weight $-2k$.  A natural $\ell$-differential operator $D$ \emph{recovers $L$} if
 \begin{enumerate}
  \item for all Riemannian manifolds of dimension $n>2k$,
  \begin{equation}
   \label{eqn:recover_noncritical}
   D(1,\dotsc,1) = \left(\frac{n-2k}{\ell+1}\right)^\ell L;
  \end{equation}
  \item for all Riemannian manifolds $(M,g)$ of dimension $n=2k$,
  \begin{equation}
   \label{eqn:recover_critical}
   \left.\frac{1}{\ell!}\frac{\partial^\ell}{\partial t^\ell}\right|_{t=0} e^{nt\Upsilon}L^{e^{2t\Upsilon}g} = D(\Upsilon,\dotsc,\Upsilon)
  \end{equation}
  for all $\Upsilon\in C^\infty(M)$.
 \end{enumerate}
\end{defn}

Putting these definitions together yields the operators we study in this article.

\begin{defn}
 \label{defn:associated}
 Let $L$ be a CVI.  A natural $\ell$-differential operator is \emph{associated to $L$} if it is conformally covariant, formally self-adjoint and recovers $L$.
\end{defn}

For operators $D$ which recover a CVI $L$, the assumptions of Definition~\ref{defn:recovers} in noncritical dimensions $n>2k$ and critical dimension $n=2k$ are closely related.  On the one hand, the bidegree is determined by $L$, and hence~\eqref{eqn:recover_noncritical} is equivalent to
\begin{equation}
 \label{eqn:pre_branson_trick}
 \left(\frac{n-2k}{\ell+1}\right)^\ell e^{\frac{n\ell+2k}{\ell+1}t\Upsilon}L^{e^{2t\Upsilon}g} = D\left(e^{\frac{n-2k}{\ell+1}t\Upsilon},\dotsc,e^{\frac{n-2k}{\ell+1}t\Upsilon}\right)
\end{equation}
for all smooth functions $\Upsilon$ and all $t\in\bR$.  Taking $j$, $1\leq j\leq\ell$, derivatives in $t$ and evaluating at zero implies that there is a $j$-differential operator $L_j^{\ell+1}$ such that
\[ D(u_1,\dotsc,u_j,1,\dotsc,1) = \frac{(\ell-j)!}{\ell!}\left(\frac{n-2k}{\ell+1}\right)^{\ell-j} L_j^{\ell+1}(u_1,\dotsc,u_j) \]
for all smooth functions $u_1,\dotsc,u_j$; we have normalized this so that $L_0^{\ell+1}=L$.  By treating $\frac{n-2k}{\ell+1}$ as a formal variable in~\eqref{eqn:pre_branson_trick}, we deduce that
\begin{equation}
 \label{eqn:branson_trick}
 e^{\frac{n\ell+2k}{\ell+1}\Upsilon}L^{e^{2\Upsilon g}} = \sum_{j=0}^\ell \frac{1}{j!}L_j^{\ell+1}(\Upsilon,\dotsc,\Upsilon) + O\left(\frac{n-2k}{\ell+1}\right) .
\end{equation}
Taking the limit $n\to 2k$ in~\eqref{eqn:branson_trick} yields a relationship analogous to~\eqref{eqn:Qprime}; differentiating the result of this limit yields~\eqref{eqn:recover_critical}.  That is, we can regard~\eqref{eqn:recover_critical} as a special case of~\eqref{eqn:recover_noncritical} via Branson's method of analytic continuation in the dimension~\cite{Branson1995}.  See Section~\ref{sec:defn} for an alternative explanation of the relationship between~\eqref{eqn:recover_noncritical} and~\eqref{eqn:recover_critical}.

Our main result is that given any CVI of weight $-2k$, one can find an associated $j$-differential operator for some integer $0\leq j\leq 2k-1$.  Indeed:

\begin{thm}
 \label{thm:existence}
 Let $L$ be a CVI of weight $-2k$.  There is an integer $1\leq j\leq 2k$ such that
 \begin{enumerate}
  \item there is a $(j-1)$-differential operator associated to $L$; and
  \item for any $1\leq\ell<j-1$, there does not exist an $\ell$-differential operator associated to $L$.
 \end{enumerate}
\end{thm}

We call the integer $j$ of Theorem~\ref{thm:existence} the \emph{rank} of $L$.  Note that if $L$ is a CVI of rank $j$ and $D$ is a $(j-1)$-differential operator associated to $L$, then the \emph{Dirichlet form $\kD$ associated to $D$},
\begin{equation}
 \label{eqn:dirichlet-form}
 \kD(u_1,\dotsc,u_j) := \int_M u_1\,D(u_2,\dotsc,u_j)\,\dvol_g ,
\end{equation}
is a conformally covariant symmetric $j$-linear form; i.e.\ $\kD$ is symmetric in its arguments and
\[ \kD^{e^{2\Upsilon}g}(u_1,\dotsc,u_j) = \kD^g\left( e^{\frac{n-2k}{j}\Upsilon}u_1,\dotsc,e^{\frac{n-2k}{j}\Upsilon}u_j\right) \]
for all $\Upsilon,u_1,u_2,\dotsc,u_j\in C^\infty(M)$ and all metrics $g$ on $M$, where $-2k$ is the weight of $L$.


We prove Theorem~\ref{thm:existence} by developing an algorithm which constructs a canonical formally self-adjoint conformally covariant polydifferential operator associated to a given CVI.  Executing this algorithm starting with the $\sigma_2$-curvature recovers the formula~\eqref{eqn:sigma2_operator} for its associated trilinear operator.  Combined with our previous work~\cite{CaseLinYuan2016} --- which algorithmically establishes a one-to-one correspondence between CVIs of weight $-2k$ and elements of $\mR_{2k}^n/\im\delta$, $n\geq 2k$, where $\mR_{2k}^n$ is the space of natural Riemannian scalar invariants of weight $-2k$ on $n$-dimensional manifolds and $\im\delta$ is the image of the divergence operator --- we obtain an algorithm which constructs all formally self-adjoint conformally covariant polydifferential operators associated to a CVI.


The rank of a CVI leads to two interesting observations.  First, the conformal transformation formulae for the Riemann curvature tensor and its covariant derivatives suggest that a generic CVI of weight $-2k$ should have rank $2k+1$.  Instead, Theorem~\ref{thm:existence} implies that every CVI has rank at most $2k$.  Second, the proof of Theorem~\ref{thm:existence} shows that if $L_1$ and $L_2$ are two CVIs of weight $-2k$ and rank $2k$, then there is a nontrivial linear combination of $L_1$ and $L_2$ of rank at most $2k-1$.  As we will see in Section~\ref{sec:weight}, the renormalized volume coefficient $v_k$ has rank $2k$.  Lemma~\ref{lem:kernel-pi} below then implies that, up to the addition of a multiple of $v_k$, every CVI of weight $-2k$ has rank at most $2k-1$; i.e.\ modulo $v_k$, every CVI is not maximally nonlinear.

In critical dimensions, Theorem~\ref{thm:existence} provides a new conformal primitive for any CVI.  Recall that if $L$ is a CVI of weight $-2k$ and $(M^{2k},g)$ is a closed $2k$-dimensional Riemannian manifold, then
\begin{equation}
 \label{eqn:conformal_primitive}
 \mS(e^{2u}g) := \int_0^1 \int_M u_s^\prime L(g_s)\,\dvol_{g_s}\,ds
\end{equation}
is a conformal primitive for $L$, where $g_s:=e^{2u_s}g$ is a one-parameter family of metrics connecting $g_0=g$ to $g_1=e^{2u}g$ and $u_s^\prime:=\frac{\partial}{\partial s}u_s$; typically one takes $u_s=su$ (cf.\ \cite{BrendleViaclovsky2004,CaseLinYuan2016}).  One can check that the definition of $\mS$ is independent of the chosen path, analogous to the Mabuchi functional~\cite{Mabuchi1986} in K\"ahler geometry.  Theorem~\ref{thm:existence} gives a manifestly path independent formula for $\mS$.

\begin{cor}
 \label{cor:critical_primitive}
 Let $L$ be a CVI of weight $-2k$ and let $(M^{2k},g)$ be a closed Riemannian manifold.  Then the functional~\eqref{eqn:conformal_primitive} is equivalently written as
 \[ \mS(e^{2u}g) = \sum_{j=0}^{2k-1} \frac{1}{(j+1)!}\int_M uL_j^j(u,\dotsc,u)\,\dvol , \]
 where $L_j^j$ and $\dvol$ are determined by $g$.
\end{cor}

The operators $L_j^j$ of Corollary~\ref{cor:critical_primitive} are the same as those appearing in the limit $n\to2k$ of~\eqref{eqn:branson_trick}, with the extra observation that $L_j^{\ell+1}=L_j^j$ in dimension $n=2k$; see Section~\ref{sec:operator}.  Note that Corollary~\ref{cor:critical_primitive} recovers the well-known~\cite{Branson1995} fact that
\[ \mS(e^{2u}g) := \int_{M^{2k}} \left( \frac{1}{2}uL_{2k}u + Q_{2k}u \right)\,\dvol \]
is a conformal primitive for Branson's $Q$-curvature $Q_{2k}$ of order $2k$ on closed Riemannian $2k$-dimensional manifolds, where $L_{2k}$ is the GJMS operator of order $2k$~\cite{GJMS1992}.  Corollary~\ref{cor:critical_primitive} also recovers the conformal primitive for the $\sigma_2$-curvature identified by Chang and Yang~\cite{ChangYang2003} and an alternative to the conformal primitive for the $\sigma_3$-curvature on locally conformally flat manifolds found by Branson and Gover~\cite{BransonGover2008}.

While Theorem~\ref{thm:existence} establishes a correspondence between formally self-adjoint conformally covariant polydifferential operators and CVIs, it is not so simple to explicitly write down a formula, or even an ambient formula, for the operator associated to a given CVI.  For this reason, we give four explicit constructions of formally self-adjoint conformally covariant polydifferential operators.

First, we describe the conclusion of Theorem~\ref{thm:existence} for well-studied CVIs.  It is clear that a CVI has rank $1$ if and only if it is conformally invariant, while the conformal covariance of the GJMS operators~\cite{GJMS1992} implies that Branson's $Q$-curvatures~\cite{Branson1995} all have rank $2$.  For any $k\in\bN$, the $k$-th renormalized volume coefficient, which is a CVI of weight $-2k$ (see~\cite{ChangFang2008,Graham2009}), has rank $2k$.  By using an observation of Case and Wang~\cite{CaseWang2016s}, we give a formula for the operator associated to $v_k$ on any flat manifold; see Section~\ref{sec:weight}.  When $k\leq3$, this expression also gives an ambient formula for the operator associated to $v_k$; see Section~\ref{sec:ambient} for details and a brief discussion of the case $k\geq4$.

Second, we give a two-parameter family of formally self-adjoint conformally covariant tridifferential operators of order four which are homogeneous of degree $-6$ in the metric; that is, all members of our family have the property that the linear operator obtained by fixing all but one of the inputs is fourth-order.  These operators are defined using the ambient metric and lead to the following nice basis for the four-dimensional space of CVIs of weight $-6$ on locally conformally flat manifolds:

\begin{thm}
 \label{thm:rank4}
 Let $(M^n,g)$, $n\not=1,2,4$, be a Riemannian manifold and define
 \begin{align*}
  I_1 & = -\Delta J^2 + \frac{n-6}{3}J^3 + \frac{2(n+2)^2}{n-2}v_3 , \\
  I_2 & = -\Delta\sigma_2 - \delta\left(T_1(\nabla J)\right) + (n-6)J\sigma_2 + \frac{3(n^2+8n-4)}{2(n-2)}v_3 ,
 \end{align*}
 where $P$ is the Schouten tensor, $J=\tr P$ is its trace, $\sigma_2=\frac{1}{2}(J^2-\lv P\rv^2)$ is the $\sigma_2$-curvature, and $T_1:=Jg-P$ is the first Newton tensor.  Then $I_1,I_2$ are CVIs of weight $-6$ and rank $4$.  Moreover, $\{Q_6,I_1,I_2,v_3\}$ forms a basis for the space of CVIs of weight $-6$ on locally conformally flat manifolds.
\end{thm}

Recall that the space of CVIs of weight $-6$ is ten-dimensional, but the subspace of CVIs which vanish on any locally conformally flat manifold is six-dimensional~\cite{CaseLinYuan2016}.  We find the invariant $I_2$ of Theorem~\ref{thm:rank4} particularly interesting as a generalization of the $\sigma_2$- and $Q_4$-curvatures, in the sense that $I_2$ (resp.\ $Q_4$) is the Laplacian of the rank $4$ CVI $\sigma_2$ (resp.\ rank $2$ CVI $\sigma_1$) plus lower order terms.  See Section~\ref{sec:trilinear} for details.

Third, we identify a special subfamily $D_{2k}$ of the conformally covariant bidifferential operators found by Ovsienko and Redou~\cite{OvsienkoRedou2003} which are likely to be formally self-adjoint.  Indeed, we give a new formula for these operators which are tangential and manifestly formally self-adjoint in the ambient space of any Riemannian manifold of dimension $n\geq2k$, thereby finding the curved analogues of this subfamily of the Ovsienko--Redou operators.  We also verify that $D_{2k}$ is formally self-adjoint when $k\leq2$, indicating that $D_{2k}$ should be formally self-adjoint for all $k$.  See Section~\ref{sec:ovsienko_redou} for details.

Fourth, we give an ambient construction of formally self-adjoint conformally covariant $\ell$-differential operators of critical order with $\ell\equiv 3\mod 4$ or, assuming the formal self-adjointness of the curved Ovsienko--Redou operators, $\ell\equiv 5\mod 6$.  These constructions complement Theorem~\ref{thm:existence} by illustrating the wide variety of ranks and orders of CVIs.  Our construction relies in a crucial way on the fact that formally self-adjoint conformally covariant operators act on densities of weight zero in the critical dimension; the generalization to noncritical dimensions involves a messy combinatorial problem which is better considered elsewhere.  See Section~\ref{sec:comments} for details.

This article is organized as follows:  In Section~\ref{sec:bg} we recall some necessary facts about CVIs.  In Section~\ref{sec:defn} we define polydifferential operators and discuss some of their important algebraic properties.  In Section~\ref{sec:operator} we describe our algorithm for constructing operators associated to a CVI and use it to prove Theorem~\ref{thm:existence}.  In Section~\ref{sec:weight} we discuss the examples of conformal invariants, Branson's $Q$-curvatures, and the renormalized volume coefficients.  In Section~\ref{sec:ambient} we give an ambient formula for the operators associated to the $\sigma_1$-, $\sigma_2$-, and $v_3$-curvatures.  In Section~\ref{sec:trilinear} we describe our family of formally self-adjoint conformally covariant tridifferential operators of total order six and prove Theorem~\ref{thm:rank4}.  In Section~\ref{sec:ovsienko_redou} we construct the curved Ovsienko--Redou operators and discuss their formal self-adjointness.  In Section~\ref{sec:comments} we describe formally self-adjoint conformally covariant differential operators in critical dimensions which take a variety of ranks and total orders.

%% file: bg.tex
\section{Background}
\label{sec:bg}

Our proof of Theorem~\ref{thm:existence} relies on the infinitesimal characterization of CVIs~\cite{BransonGover2008}.  To explain this we first recall the definition~\cite{CaseLinYuan2016} of a CVI.

\begin{defn}
 \label{defn:cvi}
 A \emph{CVI of weight $-2k$} is a natural Riemannian scalar invariant $L$ which is homogeneous of degree $-2k$ and has the property that there is a Riemannian functional $\mS\colon\kC\to\bR$ such that
 \[ \left.\frac{\partial}{\partial t}\right|_{t=0}\mS(e^{2t\Upsilon}g) = \int_M L_g\Upsilon\,\dvol_g \]
 for all conformal manifolds $(M^n,\kC)$, all metrics $g\in\kC$, and all $\Upsilon\in C^\infty(M)$.
\end{defn}

Recall here that, in terms of Definition~\ref{defn:natural}, a natural Riemannian scalar invariant is a natural $0$-differential operator.  The functional $\mS$ in Definition~\ref{defn:cvi} is a \emph{conformal primitive for $L$}.  We say a functional $\mS$ is \emph{Riemannian} if $\mS(\phi^\ast g)=\mS(g)$ for all Riemannian manifolds $(M,g)$ and all diffeomorphisms $\phi\colon M\to M$.

Denote by $\kM$ the space of Riemannian metrics on a given smooth manifold.  The existence of a conformal primitive can be characterized in terms of the infinitesimal conformal variation of $L$:

\begin{lem}[{\cite[Lemma~2]{BransonGover2008}}]
 \label{lem:cvi_linearization}
 Let $L$ be a natural Riemannian scalar invariant which is homogeneous of degree $-2k$.  Given $g\in\kM$, define $S\colon C^\infty(M)\to C^\infty(M)$ by
 \[ S(\Upsilon) = \left.\frac{\partial}{\partial t}\right|_{t=0}e^{2kt\Upsilon}L(e^{2t\Upsilon}g) . \]
 Then $S(1)=0$.  Moreover, $S$ is formally self-adjoint if and only if $L$ is a CVI.
\end{lem}

Our explicit constructions of formally self-adjoint conformally covariant polydifferential operators all require the Fefferman--Graham ambient space~\cite{FeffermanGraham2012}.  Here we recall the definition and some required properties of this space.

Let $(M^n,g)$, $n\geq2$, be a Riemannian manifold.  Let $\mG$ be the trivial $\bR_+$-bundle over $M$ corresponding to the conformal class of $g$; i.e.\ $\mG=\bR_+\times M$ with points $(t,x)\in\mG$ identified with the pair $(x,t^2g_x)$.  Denote by $\pi\colon\mG\to M$ the projection $\pi(t,x)=x$.  The canonical metric on $\mG$ is the metric $\boldsymbol{g}_{(t,x)}(Y,Z)=t^2g_x(\pi_\ast Y,\pi_\ast Z)$.  Note that $\boldsymbol{g}$ is degenerate.  Given $s>0$, we define the dilation $\delta_s\colon\mG\to\mG$ by $\delta_s(t,x)=(st,x)$.  Note that $\delta_s^\ast\boldsymbol{g}=s^2\boldsymbol{g}$.  Denote by $X$ the infinitesimal generator of $\delta_s$.

The \emph{ambient space} is $\cmG:=\mG\times(-1,1)$.  We denote points in $\cmG$ as triples $(t,x,\rho)$ and define $\iota\colon\mG\to\cmG$ by $\iota(t,x)=(t,x,0)$.  We extend $\delta_s$ and $X$ to $\cmG$ in the obvious way.  An \emph{ambient metric} is a Lorentzian metric $\cg$ on $\cmG$ such that
\begin{enumerate}
 \item $\delta_s^\ast\cg=s^2\cg$ for all $s>0$;
 \item $\iota^\ast\cg=\boldsymbol{g}$; and
 \item $\Ric(\cg)=O(\rho^\infty)$ if $n$ is odd; $\Ric(\cg)=O^+(\rho^{n/2-1})$ if $n$ is even.
\end{enumerate}
Here a symmetric $(0,2)$-tensor field $T$ on $\cmG$ belongs to $O^+(\rho^k)$ if $T\in O(\rho^k)$ and for each $(t,x)\in\mG$, it holds that $(\iota^\ast(\rho^{-k}T))(t,x)=\pi^\ast S$ for some symmetric $(0,2)$-tensor at $x$ such that $\tr_{g_x}S=0$.  Fefferman and Graham~\cite{FeffermanGraham2012} proved that ambient metrics exist and are unique modulo diffeomorphism fixing $\iota(\mG)$ and terms of order $O(\rho^\infty)$ if $n$ is odd or order $O^+(\rho^{n/2})$ if $n$ is even.  Indeed, they recursively constructed a formal power series solution to the equation $\Ric(\cg)=0$ to the claimed order.  For later computations we require the following information about the first few terms in the expansion of the ambient metric.

\begin{lem}
 \label{lem:ambient}
 Let $(M^n,g)$, $n\not=1,2,4$, be a Riemannian manifold.  The straight ambient extension is
 \begin{equation}
  \label{eqn:ambient_metric}
  \cg = 2\rho\,dt^2 + 2t\,dt\,d\rho + t^2g_\rho,
 \end{equation}
 where
 \begin{equation}
  \label{eqn:ambient}
  \begin{split}
   g_\rho & = g + 2\rho P + \rho^2\left(P^2-\frac{1}{n-4}B\right) + o(\rho^2) , \\
   \tr_g g_\rho & = n + 2\rho J + \rho^2\lv P\rv^2 - \frac{4}{3(n-4)}\rho^3\lp B,P\rp + O(\rho^4) ,
  \end{split}
 \end{equation}
 where $P$ is the Schouten tensor, $J$ is its trace, and $B=\delta dP + 2W\cdot P$ is the generalized Bach tensor.
\end{lem}

Lemma~\ref{lem:ambient} is a compilation of multiple facts from~\cite{FeffermanGraham2012}.  The coefficients of $\rho$ and $\rho^2$ in~\eqref{eqn:ambient} come from~\cite[Equation (3.6)]{FeffermanGraham2012} and~\cite[Equation (3.18)]{FeffermanGraham2012}, respectively.  The order of the next term in the first expansion in~\eqref{eqn:ambient} allows for the possibility that $n=6$, where there is a term at order $\rho^2\log\rho$ corresponding to the obstruction.  The second expansion in~\eqref{eqn:ambient} uses the fact that the obstruction tensor is trace-free~\cite[Theorems 3.8 and 3.10]{FeffermanGraham2012}.

We require two consequences of Lemma~\ref{lem:ambient}.  First, $\cnabla X = \cg$, and in particular
\begin{equation}
 \label{eqn:Xcurv}
 R(X,\cdot,\cdot,\cdot)=0 .
\end{equation}
Second, we have the following local formula for certain ambient operators.

\begin{lem}
 \label{lem:local_ambient_formulas}
 Let $(M^n,g)$, $n\not=1,2,4$, be a Riemannian manifold and let $(\cmG,\cg)$ be its straight ambient extension.  Let $w,\mu\in\bR$, and $u\in C^\infty(M\times(-1,1))$, and $\cu=t^\mu u\in C^\infty(\cmG)$.  Then
 \begin{align}
  \label{eqn:gradient} \lv\cnabla t^w\rv^2 & = 0 , \\
  \label{eqn:divergence} \cdelta\bigl(\cu\,\cd t^w\bigr) & = wt^{w+\mu-2}\left(\partial_\rho u + Ju - \rho\lv P\rv^2u + \rho^2Yu\right) + o(\rho^2), \\
  \label{eqn:Laplacian} \cDelta\cu & = t^{\mu-2}\Bigl[ -2\rho\partial_\rho^2 u + (n+2\mu-2-2\rho J + 2\rho^2\lv P\rv^2)\partial_\rho u + \Delta_\rho u \\
   \notag & \qquad + \mu\left( J - \rho\lv P\rv^2 + \rho^2Y\right) u\Bigr] + o(\rho^2), \\
  \label{eqn:hessian} \cnabla^2t^w & = w(w-1)t^{w-3}dt\otimes\partial_\rho + wt^{w-2}U + o(\rho) .
 \end{align}
 Here $Y$ is the scalar
 \begin{equation}
  \label{eqn:Y} Y = \tr P^3 + \frac{1}{n-4}\lp B,P\rp,
 \end{equation}
 $\Delta_\rho$ is the Laplacian of the metric $g_\rho$,
 \begin{equation}
  \label{eqn:Deltarho-to-Delta}
  \Delta_\rho = \Delta + \rho\left[ \delta(Jg-2P)d - J\Delta \right] + o(\rho) ,
 \end{equation}
 and $U$ is the one-parameter family
 \begin{equation}
  \label{eqn:U}
  U = P - \rho\left(P^2+\frac{1}{n-4}B\right) \in \Gamma\left( T^\ast M \otimes TM\right)
 \end{equation}
 of sections of $T^\ast M\otimes TM$.
\end{lem}

\begin{proof}
 It follows immediately from Lemma~\ref{lem:ambient} that
 \begin{align}
  \label{eqn:determinant} \log\frac{\det g_\rho}{\det g} & = 2\rho J - \rho^2\lv P\rv^2 + \frac{2}{3}\rho^3Y + o(\rho^3) , \\
  \label{eqn:inverse} \cg^{-1} & = 2t^{-1}\partial_t\partial_\rho - 2\rho t^{-2}\partial_\rho\partial_\rho \\
   \notag & \quad + t^{-2}\left(g^{ij} - 2\rho P^{ij} + \rho^3\left(3P^{ik}P^j_k + \frac{1}{n-4}B^{ij}\right)\right) + o(\rho^2) .
 \end{align}
 In particular, \eqref{eqn:inverse} implies that for any $w\in\bR$, the vector field dual to $dt^w$ is $wt^{w-2}\partial_\rho$.  This yields~\eqref{eqn:gradient}.  Moreover, combining~\eqref{eqn:determinant} and~\eqref{eqn:inverse} with the coordinate formula for the inverse yields~\eqref{eqn:divergence} and~\eqref{eqn:Laplacian}.  Finally, using Lemma~\ref{lem:ambient} to compute the Christoffel symbols of $\cg$ yields~\eqref{eqn:hessian}.
\end{proof}

A natural way to describe conformally covariant operators is through conformal density bundles.  Let $(\cmG,\cg)$ be the ambient space of a Riemannian manifold $(M^n,g)$.  Given $w\in\bR$, let
\[ \cmE[w] := \left\{ u \in C^\infty(\cmG) \suchthat Xu = wu \right\} \]
denote the space of ambient functions which are homogeneous of degree $w$ with respect to the dilations $\delta_s$.  The \emph{conformal density bundle of weight $w$} is obtained by restricting $\cmE[w]$ to $\mG$,
\[ \mE[w] := \left\{ u\rv_{\mG} \suchthat u\in\cmE[w] \right\} . \]
A \emph{conformal density of weight $w$} is an element of $\mE[w]$.  Note that if $u$ is a conformal density of weight $w$, then
\begin{equation}
 \label{eqn:density_transformation}
 u\left(x,(e^{2\Upsilon}g)_x\right) = e^{w\Upsilon(x)}u(x,g_x)
\end{equation}
for all $x\in M$ and all $\Upsilon\in C^\infty(M)$.  Given $f\in C^\infty(M)$, define
\[ u(x,t^2g_x) := t^wf(x) . \]
This gives a $g$-dependent identification $C^\infty(M)\cong_g\mE[w]$.  A different element $\hg$ of the conformal class $[g]$ gives a different identification $C^\infty(M)\cong_{\hg}\mE[w]$, and the relation between the identifications $\cong_g$ and $\cong_{\hg}$ is easily determined from~\eqref{eqn:density_transformation}.

One can construct conformally covariant operators on $(M,g)$ by finding tangential operators in its ambient space.

\begin{defn}
 \label{defn:tangential}
 Let $(M^n,g)$ be a Riemannian manifold and let $(\cmG,\cg)$ be its ambient space.  A differential operator $\cL\colon\cmE[w]\to\cmE[w^\prime]$ is \emph{tangential} if the map
\[ \cmE[w] \ni u \mapsto \cL(u)\rv_{\mG} \in \mE[w^\prime] \]
depends only on $u\rv_{\mG}\in\mE[w]$.
\end{defn}

In particular, if $\cL$ is tangential, it induces a differential operator $L\colon\mE[w]\to\mE[w^\prime]$.  Let $L^g\colon C^\infty(M)\to C^\infty(M)$ be the operator determined by $L$ using the identifications $\mE[w]\cong_g C^\infty(M)$.  It follows from~\eqref{eqn:density_transformation} that
\[ L^{e^{2\Upsilon}g} = e^{w^\prime\Upsilon}\circ L^g\circ e^{-w\Upsilon} , \]
where $e^{w^\prime\Upsilon},e^{w \Upsilon}$ are regarded as multiplication operators.  In particular, $L^g$ is a conformally covariant operator.  A similar relationship holds between tangential polydifferential operators in $(\cmG,\cg)$ and conformally covariant polydifferential operators on $(M,g)$; see Section~\ref{sec:ambient}.

Since $Q:=\lv X\rv^2\in\cmE[2]$ is a defining function for $\mG$ --- that is, $Q^{-1}(0)=\mG$ and $dQ\rv_{\mG}\not=0$ --- we have an easy way to determine if an operator $\cL\colon\cmE[w]\to\cmE[w^\prime]$ is tangential.

\begin{lem}
 \label{lem:tangential}
 Let $(M^n,g)$ be a Riemannian manifold and let $(\cmG,\cg)$ be its ambient space.  A differential operator $\cL\colon\cmE[w]\to\cmE[w^\prime]$ is tangential if and only if
 \begin{equation}
  \label{eqn:cLQ}
  \cL(Qz) \equiv 0 \mod Q
 \end{equation}
 for all $z\in\cmE[w-2]$.
\end{lem}

\begin{proof}
 It is clear that if $\cL$ is tangential, then~\eqref{eqn:cLQ} holds for all $z\in\cmE[w-2]$.  Conversely, suppose that~\eqref{eqn:cLQ} holds for all $z\in\cmE[w-2]$.  Let $u_1,u_2\in\cmE[w]$ be such that $u_1\rv_\mG=u_2\rv_\mG$.  Since $Q$ is a defining function for $\mG$, it holds that $u_1-u_2=Qz$ for some $z\in\cmE[w-2]$.  Since $\cL$ is linear, we compute that
 \[ \cL(u_1)\rv_{\mG} = \cL(u_2)\rv_{\mG} + \cL(Qz)\rv_{\mG} = \cL(u_2)\rv_{\mG} . \]
 Therefore $\cL$ is tangential.
\end{proof}

Graham, Jenne, Mason and Sparling~\cite{GJMS1992} proved the commutator formula
\begin{equation}
 \label{eqn:DeltakX}
 [\cDelta^k,Q] = 2k\cDelta^{k-1}(2X + n + 4 - 2k)
\end{equation}
to show that the operator $(-\cDelta)^k$ is tangential on $\cmE\bigl[-\frac{n-2k}{2}\bigr]$.  In Sections~\ref{sec:trilinear}--\ref{sec:comments}, we show that certain ambient operators are tangential by using~\eqref{eqn:cLQ}, \eqref{eqn:DeltakX}, the fact $\cnabla Q=2X$, and its consequences
\begin{align}
 \label{eqn:divf} \cdelta\left(u\,\cd(Qv)\right) & = Q\cdelta\left(u\,\cd v\right) + 2(v\,Xu + 2u\,Xv + (n+2)uv), \\
 \label{eqn:divQ} \cdelta\left(Qu\,\cd v\right) & = Q\cdelta\left(u\,\cd v\right) + 2u\,Xv .
\end{align}
for all $u,v\in C^\infty(\cmG)$.

%% file: defn.tex
\section{Some facts about polydifferential operators}
\label{sec:defn}

In order to construct conformally covariant polydifferential operators on Riemannian manifolds, we must first introduce a number of useful concepts and notations.  We begin by discussing polydifferential operators.

Given a smooth manifold $M$, we denote by $\kM$ the space of Riemannian metrics on $M$.  Given also a nonnegative integer $k\in\bN_0$, we denote by $\Lin_k$ the space of $k$-multilinear maps $D$ from
\[ \left(C^\infty(M)\right)^k := \underbrace{ C^\infty(M) \times \dotsm \times C^\infty(M)}_{\text{$k$ copies}} \]
to $C^\infty(M)$ which act as differential operators on each of their inputs, with the convention that elements of $\Lin_0$ are smooth functions on $M$.

There are a number of ways to construct elements of $\Lin_k$.  We begin by describing two distinct ways to regard a function as a multiplication operator.

The first method is to multiply in all components.

\begin{defn}
 Let $k\in\bN_0$ and $\phi\in C^\infty(M)$.  The \emph{multiplication operator $\phi\colon\left(C^\infty(M)\right)^k\to\left(C^\infty(M)\right)^k$} is given by
 \[ \phi(u_1,\dotsc,u_k) := \left(\phi u_1, \dotsc, \phi u_k\right) \]
 for all $u_1,\dotsc,u_k\in C^\infty(M)$.
\end{defn}

Note that if $\phi,\psi\in C^\infty(M)$ and $D\in\Lin_k$, then composition yields a new operator $\phi\circ D\circ\psi\in\Lin_k$,
\begin{equation}
 \label{eqn:prepostcompose}
 \left(\phi\circ D\circ\psi\right)\left(u_1,\dotsc,u_k\right) := \phi D\left(\psi u_1, \dotsc, \psi u_k\right)
\end{equation}
for all $u_1,\dotsc,u_k\in C^\infty(M)$.

The second method is by summing over all choices of multiplying in only one component.

\begin{defn}
 The \emph{action of $\phi\in C^\infty(M)$ on $D\in\Lin_k$} is the operator $\phi\ast D\in\Lin_k$ given by
 \[ \left(\phi\ast D\right)\left(u_1,\dotsc,u_k\right) := \sum_{j=1}^k D\left(u_1, \dotsc, \phi u_j, \dotsc, u_k\right) \]
 for all $u_1,\dotsc,u_k\in C^\infty(M)$.
\end{defn}

We also need the analogue of the interior product.

\begin{defn}
 \label{defn:contraction}
 The \emph{contraction of $\phi\in C^\infty(M)$ into $D\in\Lin_k$} is the operator $D(\phi)\in\Lin_{k-1}$ given by
 \[ D(\phi)\left(u_1,\dotsc,u_{k-1}\right) := D\left(\phi,u_1,\dotsc,u_{k-1}\right) \]
 for all $u_1,\dotsc,u_{k-1}\in C^\infty(M)$.
\end{defn}

These methods can be combined.  For example, if $\phi,\psi\in C^\infty(M)$ and $D\in\Lin_k$, then one obtains elements of $\Lin_{k-1}$ by
\[ \left(\phi\ast D(\psi)\right)(u_1,\dotsc,u_{k-1}) := \sum_{j=1}^{k-1}D(\psi,u_1,\dotsc,\phi u_j,\dotsc, u_{k-1}) \]
for all $u_1,\dotsc,u_{k-1}\in C^\infty(M)$.  Note that
\begin{equation}
 \label{eqn:action-contraction}
 (\phi\ast D)(\psi) = D(\phi\psi)+\phi\ast D(\psi) .
\end{equation}

We now specialize to natural polydifferential operators in the sense of Definition~\ref{defn:natural}, and denote by $\Poly_k$ and $\Poly$ the spaces of natural $k$-differential operators and natural polydifferential operators, respectively.  Given a Riemannian manifold $(M^n,g)$ and an operator $D\in\Poly_k$, we denote by $D^g$ the element of $\Lin_k$ determined by $g$.  When the metric $g$ is clear by context, we often omit the superscript.

Recall from Definition~\ref{defn:formally_self-adjoint} that an element $D\in\Poly_k$ is formally self-adjoint if for all smooth manifolds $M$ and all $g\in\kM$, the map
\[ \bigl(C^\infty(M)\bigr)_0^{k+1} \ni (u_0,\dotsc,u_k) \mapsto \int_M u_0\,D(u_1,\dotsc,u_k)\,\dvol \]
is symmetric in its arguments, where 
\[ \bigl(C^\infty(M)\bigr)_0^{k+1} := \left\{ (u_0,\dotsc,u_k) \in \bigl(C^\infty(M)\bigr)^{k+1} \suchthat \text{$\supp\,(u_0\dotsm u_k)$ is compact} \right\} . \]
Note that $D$ is symmetric in its arguments if it is formally self-adjoint.  We denote by $\FSA_k\subset\Poly_k$ the set of all formally self-adjoint natural $k$-differential operators, and by $\FSA$ the set of all formally self-adjoint natural polydifferential operators.

Note that contraction with a constant, in the sense of Definition~\ref{defn:contraction}, yields a map from $\FSA_{k+1}$ to $\FSA_k$.

\begin{defn}
 Fix $k\in\bN_0$.  The \emph{interior operation} $I_{k+1}\colon\FSA_{k+1}\to\FSA_k$ is defined by
 \[ I_{k+1}(D)(u_1,\dotsc,u_k) := D(1,u_1,\dotsc,u_k) \]
 for all $D\in\FSA_k$ and all $u_1,\dotsc,u_k\in C^\infty(M)$.
 
 An operator $D\in\FSA_k$ \emph{annihilates constants} if $D\in\ker I_k$.
\end{defn}

The homogeneity of CVIs implies that the polydifferential operators constructed from CVIs are also homogeneous.

\begin{defn}
 An operator $D\in\FSA_k$ is \emph{homogeneous of weight $w\in\bR$} if $D^{c^2g}=c^wD^g$ for all $g\in\kM$ and all constants $c>0$.  We denote by
 \[ \FSA_k^{w} = \left\{ D\in\FSA_k \suchthat \text{$D$ is homogeneous of weight $w$} \right\} . \]
\end{defn}

Note that the interior operation restricts to a map from $\FSA_{k+1}^w$ to $\FSA_k^w$.

The weight provides a useful organizational tool for classifying natural Riemannian polydifferential operators.  Of particular importance are the facts that the weight of a nonzero element $D\in\Poly_k$ is a nonpositive even integer and that there is a sharp upper bound on the weight of a nonzero operator which annihilates constants.

\begin{lem}
 \label{lem:kernel-pi}
 Fix $k\in\bN_0$ and $w\in\bR$.
 \begin{enumerate}
  \item If $w\not\in -2\bN_0$, then $\FSA_k^{w}=\{0\}$.
  \item If $w\geq -k$, then $\FSA_k^{w}\cap\ker I_k=\{0\}$.
  \item If $k$ is odd, then $\FSA_k^{-k-1}\cap\ker I_k$ is spanned by the operator
  \begin{equation}
   \label{eqn:top-degree-operator}
   L(u_1,\dotsc,u_k) := \sum_{\sigma\in S_k} \delta\left(\lp\nabla u_{\sigma(1)},\nabla u_{\sigma(2)}\rp\dotsb\lp\nabla u_{\sigma(k-2)},\nabla u_{\sigma(k-1)}\rp\,du_{\sigma(k)}\right) .
  \end{equation}
 \end{enumerate}
\end{lem}

\begin{remark}
 Here $\lp \cdot , \cdot \rp$ denotes the Riemannian metric $g$, and hence $\nabla u$ denotes the gradient (with respect to $g$) of $u$.  While \eqref{eqn:top-degree-operator} can be written purely in terms of the exterior derivatives $du$ and the induced metric on $T^\ast M$, we have opted to present \eqref{eqn:top-degree-operator} in its current form to make it more clear what is the one-form whose divergence is being taken.
\end{remark}

\begin{proof}
 The Levi-Civita connection and the Riemann curvature tensor $R_{ijk}{}^l$ both have weight $0$, while the inverse metric $g^{ij}$ has weight $-2$.  It follows from the definition of naturality that if $\FSA_k^{w}\not=\{0\}$, then $w$ must be a nonpositive even integer.  This proves the first assertion.

 Suppose now that $D\in\FSA_k^{w}\cap\ker I_k$ for some $w\geq -k$.  In particular, $D(1)=0$.  Since $D$ is symmetric, it must factor through the exterior derivative; i.e.
 \[ D(u_1,\dotsc,u_k) = B(du_1,\dotsc,du_k) \]
 for some polydifferential operator $B\colon\bigl(\Omega^1M\bigr)^k\to C^\infty(M)$ on one-forms.  Since $D$ is formally self-adjoint, it must also be in the image of the divergence.  Hence there is a multilinear operator $A\colon\bigl(\Omega^1M\bigr)^k\to\Omega^1M$ such that
 \[ D(u_1,\dotsc,u_k) = \delta\left(A(du_1,\dotsc,du_k)\right) . \]
 Since the exterior derivative $d$ is homogeneous of weight $0$ and the divergence $\delta$ is homogeneous of weight $-2$, it follows that $A$ is homogenous of weight $w+2$.  The naturality of $D$ implies that $A$ is natural.  Since $A$ maps $\left(\Omega^1M\right)^k$ to $\Omega^1M$, if $A$ is nonzero, then $A$ involves at least $\lfloor k/2\rfloor$ contractions.  Thus if $A\not=0$, then $w+2\leq -2\lfloor k/2\rfloor$.  Using the fact $2\lfloor k/2\rfloor\geq k-1$, we obtain two consequences:
 \begin{enumerate}
  \item If $w\geq -k$, then $A=0$.  This implies the second assertion.
  \item If $k$ is odd and $w=-k-1$, then $A$ involves exactly $\lfloor k/2\rfloor = (k-1)/2$ contractions, and hence $A$ is a constant multiple of the symmetrization of $g^{i_1i_2}\dotsb g^{i_{k-2}i_{k-1}}$.  This implies the third assertion. \qedhere
 \end{enumerate}
\end{proof}

As noted in the introduction, any conformally covariant natural polydifferential operator is homogeneous, and moreover, the degree of homogeneity is determined by the bidegree.  Conversely, the bidegree of any conformally covariant formally self-adjoint polydifferential operator is determined by the degree of homogeneity.

\begin{lem}
 \label{lem:find_weights}
 Let $D\in\FSA_\ell^{-2k}$ be conformally covariant of bidegree $(a,b)$.  Then
 \[ a = \frac{n-2k}{\ell+1} \quad\text{and}\quad b = \frac{n\ell+2k}{\ell+1} , \]
 where $n=\dim M$.
\end{lem}

\begin{proof}
 Let $c>0$ be a constant.  Applying~\eqref{eqn:conformally_covariant} yields
 \[ D^{c^2g} = c^{a\ell-b}D^g . \]
 Hence $b=a\ell+2k$.  Integrating~\eqref{eqn:conformally_covariant} against $u_0\dvol_{e^{2\Upsilon}g}$ yields
 \[ \int_M u_0 D^{e^{2\Upsilon}g}(u_1,\dotsc,u_\ell)\dvol_{e^{2\Upsilon}g} = \int_M e^{(n-b)\Upsilon}u_0 D^g\left(e^{a\Upsilon}u_1,\dotsc,e^{a\Upsilon}u_\ell\right)\,\dvol_g \]
 for all $u_1,\dotsc,u_\ell\in C^\infty(M)$.  Since $D$ is formally self-adjoint, both sides are symmetric in all variables, and hence $n-b=a$.  The conclusion readily follows.
\end{proof}

We conclude by observing that an element of $\FSA$ is conformally covariant if and only if it is infinitesimally conformally covariant.  This generalizes an observation of Branson~\cite{Branson1995} for $1$-differential operators.

\begin{lem}
 \label{lem:infinitesimal_covariant_operator}
 An operator $D\in\Poly$ is conformally covariant of bidegree $(a,b)$ if and only if for every metric $g$ and every $\Upsilon\in C^\infty(M)$, it holds that
 \begin{equation}
  \label{eqn:infinitesimal_covariant_operator}
  \left.\frac{\partial}{\partial t}\right|_{t=0} \left(e^{bt\Upsilon}\circ D^{e^{2t\Upsilon}g}\circ e^{-at\Upsilon}\right) = 0 .
 \end{equation}
\end{lem}

\begin{proof}
 If $D$ is conformally covariant of bidegree $(a,b)$, then~\eqref{eqn:conformally_covariant} yields
 \[ e^{bt\Upsilon}\circ D^{e^{2t\Upsilon}g}\circ e^{-at\Upsilon} = D^g \]
 for all $t\in\bR$.  In particular, \eqref{eqn:infinitesimal_covariant_operator} holds.  Conversely, if~\eqref{eqn:infinitesimal_covariant_operator} holds, then integrating~\eqref{eqn:infinitesimal_covariant_operator} over $t\in[0,1]$ yields~\eqref{eqn:conformally_covariant}.
\end{proof}

%% file: operator.tex
\section{Constructing conformally covariant operators from CVIs}
\label{sec:operator}

We now turn to the construction of the conformally covariant polydifferential operator $D$ associated to a given CVI $L$ whose existence is asserted in Theorem~\ref{thm:existence}.  The main idea is that, by Lemma~\ref{lem:infinitesimal_covariant_operator}, the existence of such an operator implies that the function $f(t):= L^{e^{2t\Upsilon}g}$ lies in the kernel of a differential operator of order equal to the rank of $D$.  The derivatives of $f$ are natural Riemannian operators associated to $L$, while the above differential operator effectively determines $D$ in terms of these derivatives.  However, this approach provides neither a clear proof of the formal self-adjointness of $D$ nor a way to compute the rank.  Both of these defects are addressed through the following definition:

\begin{defn}
 \label{defn:Ljell}
 Let $L$ be a CVI of weight $-2k$ and fix $\ell\in\bN$.  Given $j\in\bN_0$, define $L_j^\ell\in\Poly_j$ recursively by $L_0^\ell=L$ and
 \[ L_{j+1}^\ell(u) = \left.\frac{\partial}{\partial t}\right|_{t=0} \biggl[ e^{2ktu}\left(L_j^\ell\right)^{e^{2tu}g} \biggr] + \frac{n-2k}{\ell}\left((\ell-1)uL_j^\ell - u\ast L_j^\ell\right) \]
 for all $j\geq0$ and $u\in C^\infty(M^n)$, where equality is understood as equality of elements of $\Poly_j$.
\end{defn}

Equivalently, Definition~\ref{defn:Ljell} defines $L_j^\ell$ recursively by $L_0^\ell=L$ and
\begin{equation}
 \label{eqn:Ljell_defn_consequence}
 L_{j+1}^\ell(u) = \left.\frac{\partial}{\partial t}\right|_{t=0} \left[ e^{\frac{n(\ell-1)+2k}{\ell}tu}\circ \left(L_j^\ell\right)^{e^{2tu}g}\circ e^{-\frac{n-2k}{\ell}tu} \right] .
\end{equation}
Thus we are associating a family of natural Riemannian operators to $L$ by taking first-order derivatives, though these are taken by composing with multiplication operators which depend on the dimension $n$, the weight $-2k$, and the expected rank $\ell$ of the multilinear operator.

The following lemma verifies our claim that the operators $L_j^\ell$ are natural.  Additionally, it uses the homogeneity of $L$ to relate $L_j^\ell$ to the image of $L_{j+1}^\ell$ under the interior operation $I_{j+1}$.

\begin{lem}
 \label{lem:Lj_simple_properties}
 Let $L$ be a CVI of weight $-2k$ and fix $\ell\in\bN$.  For any $j\in\bN_0$, it holds that $L_j^\ell$ is a homogeneous natural Riemannian operator of weight $-2k$ and
 \begin{equation}
  \label{eqn:piLj}
  L_{j+1}^\ell(1) = \frac{(n-2k)(\ell-j-1)}{\ell}L_j^\ell .
 \end{equation}
\end{lem}

\begin{proof}
 Let $\ell\in\bN$.  Since $L$ is a CVI of weight $-2k$, we see that $L_0^\ell$ is a homogeneous natural Riemannian operator of weight $-2k$.  Suppose that $L_j^\ell$ is a homogeneous natural Riemannian operator of weight $-2k$ for some $j\in\bN_0$.  Since the conformal variations of the Levi-Civita connection and the Riemann curvature tensor are natural Riemannian operators, we conclude from Definition~\ref{defn:Ljell} that $L_{j+1}^\ell$ is a natural Riemannian operator of weight $-2k$.  We conclude by induction that $L_j^\ell$ is a natural Riemannian operator of weight $-2k$ for all $j\in\bN_0$.
 
 Taking $u=1$ in Definition~\ref{defn:Ljell} yields~\eqref{eqn:piLj}.
\end{proof}

Crucially, the operators $L_j^\ell$ in Definition~\ref{defn:Ljell} are formally self-adjoint.

\begin{prop}
 \label{prop:Lj_sym}
 Let $L$ be a CVI of weight $-2k$ and let $\ell\in\bN$.  For each nonnegative integer $j$, it holds that $L_j^\ell\in\FSA_j^{-2k}$.
\end{prop}

\begin{proof}
 Let $\ell\in\bN$.  By Lemma~\ref{lem:Lj_simple_properties}, we need only show that $L_j^\ell$ is formally self-adjoint for all $j$.  We prove this by induction.

 It trivially holds that $L_0^\ell\in\FSA_0^{-2k}$.  Since $L$ is a CVI,
 \[ L_1^\ell(u) = S(u) + \frac{(\ell-1)(n-2k)}{\ell}Lu \]
 for all $u\in C^\infty(M)$, where $S$ is the formally self-adjoint operator from Lemma~\ref{lem:cvi_linearization}.  In particular $L_1^\ell$ is formally self-adjoint.

 Suppose now that $j\in\bN$ is such that $L_{j-1}^\ell$ and $L_j^\ell$ are both formally self-adjoint.

 We first show that $L_{j+1}^\ell$ is symmetric.  Since $L_j^\ell$ is symmetric, it follows readily from Definition~\ref{defn:Ljell} that $L_{j+1}^\ell(u_0)$ is symmetric as an element of $\Lin_j$.  Thus it suffices to show that $L_{j+1}^\ell(u_0,u_1)$ is symmetric in $u_0,u_1$.  Let $u_0,u_1\in C^\infty(M)$.  Define $g_{s,t}:=e^{2(su_0+tu_1)}g$ and denote $g_s:=g_{s,0}$.  A direct computation using Definition~\ref{defn:Ljell} yields
 \begin{multline*}
  \left.\frac{\partial^2}{\partial s\partial t}\right|_{s,t=0} e^{2k(su_0+tu_1)}\left(L_{j-1}^\ell\right)^{g_{s,t}} = \left.\frac{\partial}{\partial s}\right|_{s=0} e^{2ksu_0}\biggl[ \left(L_j^\ell\right)^{g_s}(u_1) \\ - \frac{n-2k}{\ell}\left((\ell-1)u_1\left(L_{j-1}^\ell\right)^{g_s} - u_1 \ast \left(L_{j-1}^\ell\right)^{g_s}\right) \biggr] .
 \end{multline*}
 For general maps $A,B\colon\bigl(C^\infty(M)\bigr)^2\to\Lin_{j-1}$, we write $A(u_0,u_1)\equiv B(u_0,u_1)$ if $A(u_0,u_1)-B(u_0,u_1)$ is symmetric in $(u_0,u_1)$.  Using Definition~\ref{defn:Ljell} again, we compute that
 \begin{multline*}
  \left.\frac{\partial}{\partial s}\right|_{s=0} e^{2ksu_0}\biggl[ (\ell-1)u_1\left(L_{j-1}^\ell\right)^{g_s} - u_1\ast\left(L_{j-1}^\ell\right)^{g_s} \biggr] \\ \equiv (\ell-1)u_1\left(L_j^\ell\right)^g(u_0) - u_1 \ast (L_j^\ell)^g(u_0) .
 \end{multline*}
 Combining the previous two displays yields
 \[ \left.\frac{\partial^2}{\partial s\partial t}\right|_{s,t=0}\left(e^{2k(su_0+tu_1)}\left(L_{j-1}^\ell\right)^{g_{s,t}}\right) \equiv \left(L_{j+1}^\ell\right)^g(u_0,u_1) . \]
 Since the left-hand side is symmetric in $u_0,u_1$, we conclude that $L_{j+1}^\ell$ is symmetric in $u_0,u_1$.

 We now show that $L_{j+1}^\ell$ is formally self-adjoint.  Since $L_{j+1}^\ell$ is symmetric, the functional
 \begin{equation}
  \label{eqn:fsa_functional}
  (u_0,u_1,\dotsc,u_{j+1}) \mapsto \int_M u_0\left(L_{j+1}^\ell\right)^g(u_1,\dotsc,u_{j+1})\,\dvol_g
 \end{equation}
 is symmetric in $(u_1,\dotsc,u_{j+1})$.  It thus suffices to show that~\eqref{eqn:fsa_functional} is symmetric in $(u_0,u_1)$.  To that end, set $g_t=e^{2tu_2}g$.  Since $L_j^\ell$ is formally self-adjoint, we have that
 \begin{equation}
  \label{eqn:fsa1_functional}
  0 \equiv \int_M u_0 \left(L_j^\ell\right)^{g_t}(u_1)\,\dvol_{g_t}
 \end{equation}
 for all $t\in\bR$, where we regard the right-hand side of~\eqref{eqn:fsa1_functional} as a family of functionals on $\left(C^\infty(M)\right)^{j-1}$ parameterized by $t$ via the formula
 \[ (u_3,\dotsc,u_{j+1}) \mapsto \int_M u_0\left(L_j^\ell\right)^{g_t}(u_1,u_3,u_4,\dotsc,u_{j+1})\,\dvol_{g_t} . \]
 Differentiating~\eqref{eqn:fsa1_functional} at $t=0$ yields
 \begin{multline*}
  0 \equiv \int_M \biggl[ u_0\left(L_{j+1}^\ell\right)^g(u_1,u_2) \\ + \frac{n-2k}{\ell}\left(u_0u_2\left(L_j^\ell\right)^g(u_1) + u_0\left(u_2\ast\left(L_j^\ell\right)^g\right)(u_1)\right) \biggr]\,\dvol_g
 \end{multline*}
 By~\eqref{eqn:action-contraction}, we compute that
 \begin{align*}
  \int_M u_0\left(u_2\ast \bigl(L_j^\ell\bigr)^g\right)(u_1)\,\dvol_g & = \int_M \left[ u_0\bigl(L_j^\ell\bigr)^g(u_1u_2) + u_0\left(u_2\ast \bigl(L_j^\ell\bigr)^g(u_1)\right) \right]\,\dvol_g \\
  & \equiv \int_M u_0\bigl(L_j^\ell\bigr)^g(u_1u_2)\,\dvol_g \\
  & \equiv \int_M u_1u_2 \bigl(L_j^\ell\bigr)^g(u_0)\,\dvol_g ,
 \end{align*}
 where the second line follows from the self-adjointness of $L_j^\ell$.  Combining the previous two displays yields
 \[ 0 \equiv \int_M u_0 \bigl(L_{j+1}^\ell\bigr)^g(u_1)\,\dvol_g, \]
 as desired.
\end{proof}

We now characterize conformally covariant polydifferential operators associated to a CVI in terms of the operators $L_j^\ell$.  To avoid using analytic continuation in the dimension, we separately present the cases of noncritical and critical dimensions.

First we consider noncritical dimensions.

\begin{prop}
 \label{prop:invariant_operator}
 Let $L$ be a CVI of weight $-2k$.  Given $\ell\in\bN$, there is a conformally covariant operator $D\in\FSA_{\ell}^{-2k}$ on manifolds of dimension $n>2k$ such that
 \begin{equation}
  \label{eqn:covariant_operator_normalization}
  D(1,\dotsc,1) = \left(\frac{n-2k}{\ell+1}\right)^\ell L
 \end{equation}
 if and only if $L_{\ell+1}^{\ell+1}=0$.
\end{prop}

\begin{proof}
 Suppose first that $L_{\ell+1}^{\ell+1}=0$.  Set $D=\frac{1}{\ell!}L_{\ell}^{\ell+1}$.   Lemma~\ref{lem:infinitesimal_covariant_operator} and~\eqref{eqn:Ljell_defn_consequence} imply that $D$ is conformally covariant, while Proposition~\ref{prop:Lj_sym} implies that $D$ is formally self-adjoint.  Moreover, repeated applications of Lemma~\ref{lem:Lj_simple_properties} implies that~\eqref{eqn:covariant_operator_normalization} holds.

 Suppose now that $D\in\FSA_\ell^{-2k}$ is a conformally covariant operator such that~\eqref{eqn:covariant_operator_normalization} holds.  Combining Lemma~\ref{lem:find_weights} and Lemma~\ref{lem:infinitesimal_covariant_operator} yields
 \begin{equation}
  \label{eqn:invariant_assumption}
  \left.\frac{\partial}{\partial t}\right|_{t=0} e^{\frac{n\ell+2k}{\ell+1}t\Upsilon}\circ D^{e^{2t\Upsilon}g} \circ e^{-\frac{n-2k}{\ell+1}t\Upsilon} = 0 .
 \end{equation}
 By combining this with~\eqref{eqn:Ljell_defn_consequence}, we see that it suffices to show that $L_\ell^{\ell+1}=\ell!D$.  To that end, for any $1\leq j\leq\ell$, denote by $D_j$ the operator
 \[ D_j(u_1,\dotsc,u_j) := D(u_1,\dotsc,u_j,1,\dotsc,1) . \]
 Using~\eqref{eqn:invariant_assumption}, we compute that
 \[ \left.\frac{\partial}{\partial t}\right|_{t=0}e^{2ktu_0}D_j^{e^{2tu_0}g} = \frac{n-2k}{\ell+1}\Bigl[ (\ell-j)D_{j+1}(u_0) + u_0\ast D_j - \ell u_0 D_j \Bigr] . \]
 A simple induction argument starting from~\eqref{eqn:covariant_operator_normalization} thus yields
 \[ D_j = \frac{(\ell-j)!}{\ell!}\left(\frac{n-2k}{\ell+1}\right)^{\ell-j}L_j^{\ell+1} . \]
 Taking $j=\ell$ yields the desired result.
\end{proof}

Next we consider the critical dimension.

\begin{prop}
 \label{prop:critical_invariant_operator}
 Let $L$ be a CVI of weight $-2k$.  Given $\ell\in\bN$, there is a conformally covariant operator $D\in\FSA_\ell^{-2k}$ on $2k$-dimensional manifolds such that
 \begin{equation}
  \label{eqn:critical_invariant_operator_normalization}
  \left.\frac{\partial^\ell}{\partial t^\ell}\right|_{t=0} e^{2ktu}L^{e^{2tu}g} = D(u,\dotsc,u)
 \end{equation}
 if and only if $L_{\ell+1}^{\ell+1}=0$.  Moreover, if~\eqref{eqn:critical_invariant_operator_normalization} holds, then
 \begin{equation}
  \label{eqn:critical_invariance}
  e^{n\Upsilon}\left(L_j^{\ell+1}\right)^{e^{2\Upsilon}g}(u_1,\dotsc,u_j) = \sum_{s=0}^{\ell-j} \frac{1}{s!}\left(L_{j+s}^{j+s}\right)^g(u_1,\dotsc,u_j,\Upsilon,\dotsc,\Upsilon) .
 \end{equation}
 for all $0\leq j\leq \ell$, all metrics $g$, and all smooth functions $\Upsilon,u_1,\dotsc,u_j$.
\end{prop}

\begin{proof}
 Since $n=2k$, we conclude from~\eqref{eqn:Ljell_defn_consequence} that
 \begin{equation}
  \label{eqn:critical_Ljell_defn_consequence}
  \left.\frac{\partial}{\partial t}\right|_{t=0} e^{nt\Upsilon}\circ\left(L_j^s\right)^{e^{2t\Upsilon}g} = \left(L_{j+1}^s\right)^g(\Upsilon)
 \end{equation}
 for any $j,s\in\bN_0$.  A simple induction argument using~\eqref{eqn:critical_Ljell_defn_consequence} implies that $L_j^s=L_j^j$ for all $j,s\in\bN_0$.  By iterating~\eqref{eqn:critical_Ljell_defn_consequence}, we conclude that
 \begin{equation}
  \label{eqn:critical_Ljell_defn_consequence2}
  \left.\frac{\partial^s}{\partial t^s}\right|_{t=0} e^{nt\Upsilon}\circ\left(L_j^{\ell+1}\right)^{e^{2t\Upsilon}g}(u_1,\dotsc,u_j) = \left(L_{j+s}^{j+s}\right)^g(u_1,\dotsc,u_j,\Upsilon,\dotsc,\Upsilon) .
 \end{equation}
 for all $0\leq j\leq\ell+1$ and all $0\leq s\leq \ell+1-j$.

 Suppose first that $L_{\ell+1}^{\ell+1}=0$.  Set $D=L_\ell^{\ell+1}$.  Lemma~\ref{lem:infinitesimal_covariant_operator} and~\eqref{eqn:critical_Ljell_defn_consequence} imply that $D$ is conformally covariant, while Proposition~\ref{prop:Lj_sym} implies that $D$ is formally self-adjoint.  We conclude that~\eqref{eqn:critical_invariant_operator_normalization} holds by applying~\eqref{eqn:critical_Ljell_defn_consequence2} in the case $j=0$ and $s=\ell$.

 Suppose now that $D\in\FSA_\ell^{-2k}$ is a conformally covariant operator such that~\eqref{eqn:critical_invariant_operator_normalization} holds.  Then using $j=0$ and $s=\ell$ in~\eqref{eqn:critical_Ljell_defn_consequence2} implies that $D=L_\ell^{\ell+1}$.  On the other hand, Lemma~\ref{lem:find_weights} and Lemma~\ref{lem:infinitesimal_covariant_operator} imply that
 \begin{equation}
  \label{eqn:critical_invariant_assumption}
  \left.\frac{\partial}{\partial t}\right|_{t=0} e^{nt\Upsilon}\circ D^{e^{2t\Upsilon}g} = 0 .
 \end{equation}
 Applying~\eqref{eqn:critical_Ljell_defn_consequence} implies that $L_{\ell+1}^{\ell+1}=0$.

 Finally, since $L_{\ell+1}^{\ell+1}=0$, we conclude from~\eqref{eqn:critical_Ljell_defn_consequence2} that $e^{nt\Upsilon}\left(L_j^{\ell+1}\right)^{e^{2t\Upsilon}g}(u_1,\dotsc,u_j)$ is a polynomial in $t$ of degree $\ell-j$.  Integrating~\eqref{eqn:critical_Ljell_defn_consequence2} then yields~\eqref{eqn:critical_invariance}.
\end{proof}

It follows from Proposition~\ref{prop:invariant_operator} and Proposition~\ref{prop:critical_invariant_operator} that a conformally covariant polydifferential operator can be associated to a given CVI $L$ if it is known that $L_{\ell}^{\ell}=0$ for some $\ell\in\bN$.  The existence of such an $\ell$ is guaranteed by the following lemma. 

\begin{lem}
 \label{lem:finite_rank}
 Let $L$ be a CVI of weight $-2k$.  Then $L_{2k}^{2k}=0$.
\end{lem}

\begin{proof}
 Lemma~\ref{lem:Lj_simple_properties} and Proposition~\ref{prop:Lj_sym} together imply that $L_{2k}^{2k}\in\FSA_{2k}^{-2k}\cap\ker I_{2k}$.  We conclude from Lemma~\ref{lem:kernel-pi} that $L_{2k}^{2k}=0$.
\end{proof}

Lemma~\ref{lem:finite_rank} guarantees that the rank, as defined below, is finite.

\begin{defn}
 \label{defn:rank}
 The \emph{rank} of a CVI $L$ of weight $-2k$ is the integer
 \[ r := \inf \left\{ \ell\in\bN \suchthat L_\ell^\ell = 0 \right\} . \]
\end{defn}

Note that a CVI $L$ of weight $-2k$ has rank $r\leq 2k$ and that $r=1$ if and only if $L$ is pointwise conformally invariant.

We now reformulate and prove Theorem~\ref{thm:existence}.

\begin{thm}
 \label{thm:rank_existence}
 Let $L$ be a CVI of weight $-2k$ and rank $r$.  Then there is an $(r-1)$-differential operator associated to $L$, in the sense of Definition~\ref{defn:associated}.  Moreover, for any integer $0\leq j<r-1$, there is no $j$-differential operator associated to $L$.
\end{thm}

\begin{proof}
 Observe that Proposition~\ref{prop:invariant_operator} and Proposition~\ref{prop:critical_invariant_operator} together state that there exists a formally self-adjoint, conformally covariant $\ell$-linear operator associated to $L$ if and only if $L_{\ell+1}^{\ell+1}=0$.  The conclusion follows immediately from Definition~\ref{defn:rank}.
\end{proof}

We conclude this section by identifying a new conformal primitive for any CVI in the case of critical dimension.

\begin{proof}[Proof of Corollary~\ref{cor:critical_primitive}]
 It follows from~\eqref{eqn:critical_invariance} that
 \begin{align*}
  \mS(e^{2u}g) & = \sum_{j=0}^{2k-1}\frac{1}{j!}\int_0^1\int_M \dot u_tL_j^j(u_t,\dotsc,u_t)\,\dvol\,dt \\
  & = \sum_{j=0}^{2k-1} \frac{1}{(j+1)!}\int_0^1\int_M \frac{\partial}{\partial t}\left(u_tL_j^j(u_t,\dotsc,u_t)\right)\,\dvol\,dt \\
  & = \sum_{j=0}^{2k-1} \frac{1}{(j+1)!}\int_M uL_j^j(u,\dotsc,u)\,\dvol,
 \end{align*}
 where the second equality uses the symmetry of $L_j^j$ and the third equality uses $(u_0,u_1)=(0,u)$.
\end{proof}

%% file: weight.tex
\section{A brief discussion of known examples}
\label{sec:weight}

In this section we discuss three special families of CVIs which are of particular interest in the literature: pointwise conformal invariants, Branson's $Q$-curvatures~\cite{Branson1995}, and the renormalized volume coefficients~\cite{Graham2000}.

\input{ex1}
\input{ex2}
\input{exk}

%% file: ex1.tex
\subsection{Conformal invariants}
\label{sec:weight/1}

Let $L$ be a CVI of weight $-2k$.  By definition, $L$ has rank $1$ if and only if it is conformally invariant.  The Dirichlet form~\eqref{eqn:dirichlet-form} associated to $L$ recovers the well-known conformally invariant map
\[ \mE[2k-n] \ni u \mapsto \int_{M^n} Lu\,\dvol . \]

%% file: ex2.tex
\subsection{Branson's $Q$-curvature}
\label{sec:weight/2}

For each $k\in\bN$, the GJMS construction~\cite{GJMS1992} produces a conformally covariant operator $L_{2k}$ of weight $-2k$ on any Riemannian manifold of dimension $n\geq 2k$.  Moreover, these operators are formally self-adjoint~\cite{FeffermanGraham2013,GrahamZworski2003,Juhl2013}.  It follows from Theorem~\ref{thm:rank_existence} that the CVI $Q_{2k}:=\frac{2}{n-2k}L_{2k}(1)$, defined in dimension $n\geq2k$ by analytic continuation in the dimension~\cite{Branson1995} (see also~\cite{FeffermanGraham2002}), associated to $L_{2k}$ is a CVI of weight $-2k$ and rank $2$.  Moreover, the Dirichlet form~\eqref{eqn:dirichlet-form} associated to $Q_{2k}$ is the energy of $L_{2k}$.

%% file: exk.tex
\subsection{Renormalized volume coefficients}
\label{sec:weight/renormalized}

Recall that given any Riemannian manifold $(M^n,g)$, there is a \emph{Poincar\'e metric}
\[ g_+ = r^{-2}\left( dx^2 + g_r\right) \]
on $X:=M\times[0,\varepsilon)$, where $g_r$ is a one-parameter family of Riemannian metrics on $M\approx M\times\{0\}$ such that $g_0=g$ and $\Ric(g_+)+ng_+=O(r^{n-2})$ and $\tr_g \iota^\ast\left(\Ric(g_+)+ng_+\right)=O(r^n)$ for $\iota\colon M\to X$ the inclusion map.  In fact, the metric $g_+$ depends only on the conformal class of $g$ and is uniquely determined up to order $r^n$ modulo diffeomorphism; see~\cite{FeffermanGraham2012} for details.  Moreover, the expansion
\begin{equation}
 \label{eqn:poincare_expansion}
 g_r = g_{(0)} + r^2g_{(2)} + r^4g_{(4)} + \dotsm
\end{equation}
of $g_r$ near $r=0$ is even up to order $n$, the terms $g_{(2k)}$, $0\leq k< n/2$, are locally determined by $g$, and $\tr_g \partial_r^ng_r=0$.  This is related to~\eqref{eqn:ambient_metric} via the change of variables $\rho=-r^2/2$.

It follows from~\eqref{eqn:poincare_expansion} and our discussion of its asymptotics that
\[ \left(\frac{\det g_r}{\det g}\right)^{1/2} = \sum_{k=0}^{\lfloor n/2\rfloor} (-2)^{-k}v_kr^{2k} + o(r^{n}) , \]
where each of the terms $v_k$, $0\leq k\leq n/2$, is locally determined by $g$.  These are the \emph{renormalized volume coefficients}~\cite{Graham2000}.  It is known~\cite{ChangFang2008,Graham2009} that the $k$-th renormalized volume coefficient $v_k$ is a CVI of weight $-2k$ and that if $k\leq 2$ or $g$ is locally conformally flat, then $v_k$ and $\sigma_k$ agree.  Note that
\[ v_3 = \frac{1}{6}J^3 - \frac{1}{2}J\lv P\rv^2 + \frac{1}{3}\tr P^3 + \frac{1}{3(n-4)}\lp B,P\rp . \]

Fix $k\in\bN$ and let $(M^n,g)$, $n>2k$, be a Riemannian manifold.  Case and Wang~\cite{CaseWang2016s} showed that if $k\leq 2$ or $g$ is locally conformally flat, then there is a formally self-adjoint conformally covariant polydifferential operator $L_{2k}\colon\bigl(C^\infty(M)\bigr)^{2k-1}\to C^\infty(M)$ such that
\begin{equation}
 \label{eqn:casewang_definition}
 L_{2k}(u,\dotsc,u) = \left(\frac{n-2k}{2k}u\right)^{2k-1}u^{\frac{4k^2}{n-2k}}\sigma_k^{g_u}
\end{equation}
for any positive $u\in C^\infty(M)$, where $g_u:=u^{\frac{4k}{n-2k}}g$.  Indeed, one simply defines $L(u)$ by the right-hand side of~\eqref{eqn:casewang_definition} and checks that the result is homogeneous of degree $2k-1$ in $u$.  In the remainder of this section, we deduce a formula for $L_{2k}$ on flat manifolds which is manifestly formally self-adjoint.  In Section~\ref{sec:ambient} we show, among other things, that this gives rise to the operator associated to $v_k$ when $k\leq3$.

Before giving the expression for $L_{2k}$, we first identify a useful basis of formally self-adjoint polydifferential operators.  These operators are specified in terms of the $j$-Hessian $\sigma_j(\nabla^2u^2)$, $j\leq k-1$, rather than the $\sigma_j$-curvatures of $g_u$, because of the the nice divergence structure of the $j$-Hessian in flat manifolds~\cite{Reilly1973}.

In order to more succinctly write this computation, set
\begin{equation}
 \label{eqn:polarized_defn}
 \begin{split}
 \sigma_j(u_1,\dotsc,u_{2j}) & := \frac{1}{j!}\delta_{i_1\dotsm i_j}^{\ell_1\dotsm\ell_j}(u_1u_2)_{\ell_1}^{i_1}\dotsm(u_{2j-1}u_{2j})_{\ell_j}^{i_j} , \\
 \left(T_j(u_1,\dotsc,u_{2j})\right)_i^\ell & := \frac{1}{j!}\delta_{ii_1\dotsm i_j}^{\ell\ell_1\dotsm\ell_j}(u_1u_2)_{\ell_1}^{i_1}\dotsm(u_{2j-1}u_{2j})_{\ell_j}^{i_j} , \\
 N_j(u_1,\dotsc,u_{2j}) & := \lp\nabla u_1,\nabla u_2\rp\dotsm\lp\nabla u_{2j-1},\nabla u_{2j}\rp ,
 \end{split}
\end{equation}
where
\[ \delta_{i_1\dotsm i_j}^{\ell_1\dotsm\ell_j} = \begin{cases}
  1, & \text{if $\ell_1,\dotsm,\ell_j$ are distinct and an even permutation of $i_1,\dotsm,i_j$}, \\
  -1, & \text{if $\ell_1,\dotsm,\ell_j$ are distinct and an odd permutation of $i_1,\dotsm,i_j$}, \\
  0, & \text{otherwise}
 \end{cases} \]
is the generalized Kronecker delta.  Note that, up to a multiplicative constant, $\sigma_j(u_1,\dotsc,u_{2j})$ is a partial polarization of the $j$-Hessian of $u^2$, that $T_j(u_1,\dotsc,u_{2j})$ is a partial polarization of the $j$-th Newton tensor of the Hessian of $u^2$, and that $N_j$ is a partial polarization of $\lv\nabla u\rv^{2j}$.

\begin{lem}
 \label{lem:vk_basis}
 Let $j,k\in\bN_0$ with $j\leq k-1$ and let $(M^n,g)$ be a flat manifold.  Define $D_j^k\colon\bigl(C^\infty(M)\bigr)^{2k-1}\to C^\infty(M)$ by
 \begin{align*}
  & D_j^k(u_1,\dotsc,u_{2k-1}) = \frac{1}{(2k-1)!} \\
  & \quad \times \sum_{\sigma\in S_{2k-1}} \biggl\{ \frac{u_{\sigma(1)}}{k-j}\delta\left(T_{j-1}(u_{\sigma(2)},\dotsc,u_{\sigma(2j-1)})(\nabla N_{k-j}(u_{\sigma(2j)},\dotsc,u_{\sigma(2k-1)}))\right) \\
  & \quad - \delta\left(N_{k-j-1}(u_{\sigma(1)},\dotsc,u_{\sigma(2k-2j-2)})\sigma_j(u_{\sigma(2k-2j-1)},\dotsc,u_{\sigma(2k-2)})\,du_{\sigma(2k-1)}\right) \biggr\}
 \end{align*}
 for all $u_1,\dotsc,u_{2k-1}\in C^\infty(M)$, where $S_{2k-1}$ is the symmetric group on $2k-1$ variables.  Then $D_j^k$ is formally self-adjoint.
\end{lem}

\begin{proof}
 Let $S_{2k-1}$ be the symmetric group on $\{1,\dotsc,2k-1\}$ and let $S_{2k}$ be the symmetric group on $\{0,\dotsc,2k-1\}$.  Thus
 \[ S_{2k} = \bigsqcup_{j=0}^{2k-1} (0\; j)S_{2k-1} , \]
 where $(0\;j)\in S_{2k}$ is the transposition switching $0$ and $j$.
 On the one hand, we deduce from the symmetries of~\eqref{eqn:polarized_defn} that
 \begin{align*}
  \MoveEqLeft[2] \sum_{\sigma\in S_{2k}} \sigma_j(u_{\sigma(0)},\dotsc,u_{\sigma(2j-1)})N_{k-j}(u_{\sigma(2j)},\dotsc,u_{\sigma(2k-1)}) \\
   & = 2j\sum_{\sigma\in S_{2k-1}} \sigma_j(u_0,u_{\sigma(1)},\dotsc,u_{\sigma(2j-1)})N_{k-j}(u_{\sigma(2j)},\dotsc,u_{\sigma(2k-1)}) \\
    & \quad + 2(k-j)\sum_{\sigma\in S_{2k-1}} N_j(u_0,u_{\sigma(1)},\dotsc,u_{\sigma(2k-2j-1)})\sigma_j(u_{\sigma(2k-2j)},\dotsc,u_{\sigma(2k-1)}) .
 \end{align*}
 On the other hand, a straightforward integration-by-parts yields
 \begin{multline*}
  -\int_M u_0\delta\left(N_{k-j-1}(u_1,\dotsc,u_{2k-2j-2})\sigma_j(u_{2k-2j-1},\dotsc,u_{2k-2})\,du_{2k-1}\right) \\ = \int_M N_{k-j}(u_0,u_{2k-1},u_1,u_2,\dotsc,u_{2k-2j-2})\sigma_j(u_{2k-2j-1},\dotsc,u_{2k-2}) .
 \end{multline*}
 and
 \begin{multline*}
  \sum_{\sigma\in S_{2k-1}} \int_M u_{0}u_{\sigma(1)} \delta\left(T_{j-1}(u_{\sigma(2)},\dotsc,u_{\sigma(2j-1)})(\nabla N_{k-j}(u_{\sigma(2j)},\dotsc,u_{\sigma(2k-1)}))\right) \\ = j\sum_{\sigma\in S_{2k-1}} \int_M \sigma_j(u_0,u_{\sigma(1)},\dotsc,u_{\sigma(2j-1)}) N_{k-j}(u_{\sigma(2j)},\dotsc,u_{\sigma(2k-1)}) ,
 \end{multline*}
 where the last identity uses the fact that the Newton tensors of the $j$-Hessian are all divergence-free on flat manifolds~\cite{Reilly1973}.  Combining the above three displays yields
 \begin{multline}
  \label{eqn:Djk_polarized_energy}
  \int_M u_0D_j^k(u_1,\dotsc,u_{2k-1}) \\ = \frac{k}{(k-j)(2k)!}\sum_{\sigma\in S_{2k}}\int_M \sigma_j(u_{\sigma(0)},\dotsc,u_{\sigma(2j-1)})N_{k-j}(u_{\sigma(2j)},\dotsc,u_{\sigma(2k-1)}) .
 \end{multline}
 In particular, $D_j^k$ is formally self-adjoint.
\end{proof}

For the remainder of this section it suffices to consider $D_j^k$ on the diagonal; i.e.\ we need only the operators
\[ D_j^k(u) = u\delta\left(\lv\nabla u\rv^{2k-2j-2}T_{j-1}(\nabla^2u^2)(\nabla\lv\nabla u\rv^2)\right) - \delta\left(\lv\nabla u\rv^{2k-2j-2}\sigma_j(\nabla^2u^2)\,du\right) \]
for $0\leq j\leq k-1$.  We begin by rewriting $D_j^k$ in such a way that it is clearly a second-order polydifferential operator.

\begin{lem}
 \label{lem:Djk_rewrite}
 Let $(M^n,g)$ be a flat manifold and let $k\in\bN_0$.  Then
 \[ D_0^k(u) = -\delta\left(\lv\nabla u\rv^{2k-2}\,du\right) \]
 is the $2k$-Laplacian.  Additionally, for any $1\leq j\leq k-1$ it holds that
 \begin{align*}
  D_j^k(u) & = -\frac{2k-j-1}{2}u^{-1}\lv\nabla u\rv^{2k-2j-2}\sigma_{j+1} - (2k-j+1)u^{-1}\lv\nabla u\rv^{2k-2j}\sigma_j \\
   & \quad + (k-j-1)u^{-1}\lv\nabla u\rv^{2k-2j-4}T_{j+1}(\nabla u,\nabla u) \\
   & \quad + 4(k-j)u^{-1}\lv\nabla u\rv^{2k-2j-2}T_j(\nabla u,\nabla u) \\
   & \quad + 4(k+1-j)u^{-1}\lv\nabla u\rv^{2k-2j}T_{j-1}(\nabla u,\nabla u) ,
 \end{align*}
 where $\sigma_j:=\sigma_j(\nabla^2u^2)$ and $T_j:=T_j(\nabla^2u^2)$.
\end{lem}

\begin{proof}
 Since $T_{j-1}\circ\nabla^2u^2=\sigma_jg-T_j$ and $\nabla^2u=\frac{1}{2}u^{-1}\nabla^2u^2-u^{-1}du\otimes du$, it holds that
 \[ uT_{j-1}(\nabla\lv\nabla u\rv^2) = \sigma_j\,du - T_j(\nabla u) - 2\lv\nabla u\rv^2T_{j-1}(\nabla u) . \]
 Therefore $A_j^k(u):=u\delta\left(\lv\nabla u\rv^{2k-2j-2}T_{j-1}(\nabla\lv\nabla u\rv^2)\right)$ is such that
 \begin{align*}
  A_j^k(u) & = u\delta\left(u^{-1}\lv\nabla u\rv^{2k-2j-2}\left(\sigma_j\,du - T_j(\nabla u) - 2\lv\nabla u\rv^2T_{j-1}(\nabla u)\right)\right) \\
  & = \delta\left(\lv\nabla u\rv^{2k-2j-2}\sigma_j\,du\right) - u^{-1}\lv\nabla u\rv^{2k-2j}\sigma_j \\
   & \quad - \delta\left(\lv\nabla u\rv^{2k-2j-2}\left(T_j(\nabla u) + 2\lv\nabla u\rv^2T_{j-1}(\nabla u)\right)\right) \\
   & \quad + u^{-1}\lv\nabla u\rv^{2k-2j-2}T_j(\nabla u,\nabla u) + 2u^{-1}\lv\nabla u\rv^{2k-2j}T_{j-1}(\nabla u,\nabla u) .
 \end{align*}
 Using again the fact that $\nabla^2u=\frac{1}{2}u^{-1}\nabla^2u^2-u^{-1}du\otimes du$, we conclude that
 \begin{align*}
  \MoveEqLeft[2] \delta\left(\lv\nabla u\rv^{2k-2j-2}T_j(\nabla u)\right) \\
   & = (k-j-1)\lv\nabla u\rv^{2k-2j-4}T_j(\nabla u,\nabla\lv\nabla u\rv^2) + \lv\nabla u\rv^{2k-2j-2}\lp T_j,\nabla^2u\rp \\
   & = \frac{2k-j-1}{2}u^{-1}\lv\nabla u\rv^{2k-2j-2}\sigma_{j+1} - (k-j-1)u^{-1}\lv\nabla u\rv^{2k-2j-4}T_{j+1}(\nabla u,\nabla u) \\
    & \quad - (2k-2j-1)u^{-1}\lv\nabla u\rv^{2k-2j-2}T_j(\nabla u,\nabla u) . 
 \end{align*}
 Combining the previous two displays with the definition of $D_j^k$ yields the desired result.
\end{proof}

The following properties of the polarizations of $\sigma_j$ and $T_j$ will help us to express the operator~\eqref{eqn:casewang_definition} in terms of linear combinations of the operators $D_j^k$.  Here we denote by
\[ \sigma_{k,j}(A,B) := \frac{1}{k!}\delta_{i_1\dotsm i_k}^{\ell_1\dotsm \ell_k}A_{\ell_1}^{i_1}\dotsm A_{\ell_j}^{i_j} B_{\ell_{j+1}}^{i_{j+1}} \dotsm B_{\ell_k}^{i_k} \]
the evaluation of the polarization of $\sigma_k$ at $j$ copies of $A$ and $k-j$ copies of $B$.

\begin{lem}
 \label{lem:esp_props}
 Let $V$ be an $n$-dimensional inner product space.  Given any symmetric $A,B\in\End(V)$ and any scalar $f$, it holds that
 \begin{equation}
  \label{eqn:foil}
  \sigma_k(A+fB) = \sum_{j=0}^k\binom{k}{j}f^{k-j}\sigma_{k,j}(A,B) .
 \end{equation}
 If $B=I$ is the identity element, then
 \begin{equation}
  \label{eqn:foil_identity}
  \sigma_k(A+fI) = \sum_{j=0}^k\binom{n-k+j}{j}f^j\sigma_{k-j}(A) 
 \end{equation}
 If instead $B$ has rank at most one, then
 \begin{equation}
  \label{eqn:foil_rank1}
  \sigma_k(A+fB) = \sigma_k(A) + \lp T_{k-1}(A), fB\rp . 
 \end{equation}
\end{lem}

\begin{proof}
 Equation~\eqref{eqn:foil} is an immediate consequence of the multilinearity of the polarization
 \begin{equation}
  \label{eqn:polarization}
  \sigma_k(A_1,\dotsc,A_k) := \frac{1}{k!}\delta_{i_1\dotsm i_k}^{\ell_1\dotsm \ell_k}(A_1)_{\ell_1}^{i_1}\dotsm(A_k)_{\ell_k}^{i_k}
 \end{equation}
 of $\sigma_k$.
 
 From the observation that
 \[ \delta_{i_1\dotsm i_k}^{\ell_1\dotsm\ell_k}\delta_{\ell_k}^{i_k} = (n+1-k)\delta_{i_1\dotsm i_{k-1}}^{\ell_1\dotsm \ell_{k-1}}, \]
 we conclude that $\sigma_{k,j}(A,I)=\frac{(n-j)!j!}{(n-k)!k!}\sigma_j(A)$.  Inserting this into~\eqref{eqn:foil} yields~\eqref{eqn:foil_identity}.
 
 If $B$ has rank at most one, then it is clear from~\eqref{eqn:polarization} that $\sigma_{k,j}(A,B)=0$ for all $0\leq j\leq k-2$.  Combining~\eqref{eqn:foil} and the identity $\lp T_{k-1}(A),B\rp=k\sigma_{k,k-1}(A)$ yields~\eqref{eqn:foil_rank1}.
\end{proof}

We now express $L_{2k}$ in terms of the operators $D_j^k$, $0\leq j\leq k-1$.

\begin{thm}
 \label{thm:vk_operator_flat}
 Let $(M^n,g)$ be a flat manifold.  Fix $n,k\in\bN$ and let $b_j$ be the solution of the recursive relation $b_0=0$ and
 \begin{equation}
  \label{eqn:recursive_relation}
  \frac{k+j}{2}b_{j+1} + \frac{(n-2k)(k+j+1)}{2k}b_j = \binom{n-k+j}{j} .
 \end{equation}
 Then
 \begin{equation}
  \label{eqn:vk_operator_flat}
  L_{2k} = -(-2)^{-k}\sum_{j=0}^{k-1}\left(\frac{n-2k}{2k}\right)^jb_{k-j}D_j^k .
 \end{equation}
\end{thm}

\begin{proof}
 First observe that~\eqref{eqn:foil_identity} implies that
 \begin{multline*}
  L_{2k}(u) = (-2)^{-k}u^{-1}\sum_{j=0}^k\binom{n-k+j}{j}\left(\frac{n-2k}{2k}\right)^{k-1-j} \\ \times \lv\nabla u\rv^{2j}\sigma_{k-j}\left(\nabla^2u^2-\frac{4(n-k)}{n-2k}du\otimes du\right) .
 \end{multline*}
 Applying~\eqref{eqn:foil_rank1}, we deduce that
 \begin{multline*}
  L_{2k}(u) = (-2)^{-k}u^{-1}\Biggl[ \sum_{j=0}^k\binom{n-k+j}{j}\left(\frac{n-2k}{2k}\right)^{k-1-j}\lv\nabla u\rv^{2j}\sigma_{k-j} \\ - \sum_{j=0}^{k-1}\frac{2(n-k)}{k}\binom{n-k+j}{j}\left(\frac{n-2k}{2k}\right)^{k-2-j}\lv\nabla u\rv^{2j}T_{k-1-j}(\nabla u,\nabla u) \Biggr], 
 \end{multline*}
 where we again write $\sigma_j:=\sigma_j(\nabla^2u^2)$ and $T_j:=T_j(\nabla^2u^2)$.  Using the fact that $T_0=\Id$, we conclude that
 \begin{equation}
  \label{eqn:L2k_simplify}
  \begin{split}
   L_{2k}(u) & = (-2)^{-k}u^{-1}\Biggl[ \sum_{j=0}^{k-1}\binom{n-k+j}{j}\left(\frac{n-2k}{2k}\right)^{k-1-j}\lv\nabla u\rv^{2j}\sigma_{k-j} \\
    & \quad - \sum_{j=0}^{k-2}\frac{2(n-k)}{k}\binom{n-k+j}{j}\left(\frac{n-2k}{2k}\right)^{k-2-j}\lv\nabla u\rv^{2j}T_{k-1-j}(\nabla u,\nabla u) \\
    & \quad - 2\binom{n-1}{k-1}\lv\nabla u\rv^{2k} \Biggr] . 
  \end{split}
 \end{equation}

 Now denote
 \[ A := -(-2)^{-k}\sum_{j=0}^{k-1}\left(\frac{n-2k}{2k}\right)^jb_{k-j}D_j^k . \]
 Using Lemma~\ref{lem:Djk_rewrite} we compute that
 \begin{multline*}
  A(u) = (-2)^{-k}u^{-1}\Biggl[ \sum_{j=0}^{k-1} \left(\frac{n-2k}{2k}\right)^{k-1-j}c_{j,k}^{(1)}\lv\nabla u\rv^{2j}\sigma_{k-j} \\
   - \sum_{j=0}^{k-2} \left(\frac{n-2k}{2k}\right)^{k-2-j}c_{j,k}^{(2)}\lv\nabla u\rv^{2j}T_{k-1-j}(\nabla u,\nabla u) - c_k^{(3)}\lv\nabla u\rv^{2k} \Biggr] ,
 \end{multline*}
 where
 \begin{align*}
  c_{j,k}^{(1)} & = \frac{j+k}{2}b_{j+1} + (j+k+1)\left(\frac{n-2k}{2k}\right)b_j, \\
  c_{j,k}^{(2)} & = (j+1)\left(b_{j+2} + 4\left(\frac{n-2k}{2k}\right)b_{j+1} + 4\left(\frac{n-2k}{2k}\right)^2b_j\right) , \\
  c_k^{(3)} & = (2k-1)b_k + 4k\left(\frac{n-2k}{2k}\right)b_{k-1} .
 \end{align*}
 It follows readily from~\eqref{eqn:recursive_relation} that $L_{2k}=A$.
\end{proof}

We conclude this section by proving that the renormalized volume coefficients have maximal rank.

\begin{prop}
 \label{prop:vk-maximal-rank}
 For any $k\in\bN$, the rank of the renormalized volume coefficient $v_k$ is $2k$.
\end{prop}

\begin{proof}
 Consider the $2k$-dimensional sphere $(S^{2k},g_0)$ of constant sectional curvature one.  As a conformally flat manifold, it holds that $v_k=\sigma_k$.  Let $x$ be a first spherical harmonic, where $x$ is a Cartesian coordinate on $\bR^{2k+1}\supset S^{2k}$.  Then $\nabla^2x=-xg_0$ and $\lv\nabla x\rv^2=1-x^2$.  Set $g_t=e^{2tx}g_0$.  We compute that
 \begin{align*}
  e^{2ktx}\sigma_k^{g_t} & = 2^{-k}\sigma_k\left((1+2tx-t^2+t^2x^2)g_0 + 2t^2\,dx\otimes dx\right) \\
  & = 2^{1-k}\binom{2k-1}{k-1}(1+2tx)(1+2tx-t^2+t^2x^2)^{k-1} ,
 \end{align*}
 where the second equality uses Lemma~\ref{lem:esp_props}.  It follows that
 \[ \left.\frac{\partial^j}{\partial t^j}\right|_{t=0}e^{2ktx}\sigma_k^{g_t} = 0 \]
 if and only if $j\geq 2k$.  Combining this with the case $j=0$ of~\eqref{eqn:critical_Ljell_defn_consequence2} implies that the operators $L_j^j$ associated to $L=v_k$ are nonvanishing on $(S^{2k},g_0)$ for all $j\leq 2k-1$.  It follows from Lemma~\ref{lem:finite_rank} that $v_k$ has rank $2k$, as claimed.
\end{proof}

%% file: ambient.tex
\section{The operators associated to $v_2$ and $v_3$}
\label{sec:ambient}

The operators defined in Theorem~\ref{thm:vk_operator_flat} are all formally self-adjoint conformally covariant polydifferential operators associated to the renormalized volume coefficients $v_k$ on a flat manifold.  From the explicit formula of the ambient metric for a flat manifold~\cite{FeffermanGraham2012}, one deduces that the same expressions~\eqref{eqn:vk_operator_flat} define tangential operators in the ambient space of a flat manifold.  One thus expects that the operators~\eqref{eqn:vk_operator_flat}, regarded as second-order operators in the ambient space, are always tangential.  However, these operators are only associated to the renormalized volume coefficients $v_k$ of generic manifolds when $k\leq3$.  The purpose of this section is to make these comments precise.

Observe first that the operator $L_2$ in Theorem~\ref{thm:vk_operator_flat} is
\[ L_2(u) = -\Delta u . \]
As an ambient expression, this operator is well-known~\cite{GJMS1992} to be tangential on $\cmE\bigl[-\frac{n-2}{2}\bigr]$ for the ambient space of any $n$-dimensional Riemannian manifold.  Indeed, it induces the conformal Laplacian, which is associated to $J=\sigma_1$.

Before proceeding to the cases of higher $k$, we must first make precise the relationship between conformally covariant operators on a Riemannian manifold and tangential operators in its ambient space.

\begin{defn}
 \label{defn:tangential_polydifferential}
 Let $(M^n,g)$ be a Riemannian manifold and let $(\cmG,\cg)$ be its ambient space.  A polydifferential operator $\cL\colon\bigl(\cmE[w]\bigr)^k\to\cmE[w^\prime]$ is \emph{tangential} if the map
 \[ \bigl(\cmE[w]\bigr)^k \ni (u_1,\dotsc,u_k) \mapsto \cL(u_1,\dotsc,u_k)\rv_\mG \in \mE[w^\prime] \]
 depends only on $(u_1\rv_\mG,\dotsc,u_k\rv_\mG)\in\bigl(\mE[w]\bigr)^k$.
\end{defn}

First we observe that tangential operators on the ambient space give rise to conformally covariant operators on the underlying Riemannian manifolds.

\begin{lem}
 \label{lem:tangential_polydifferential}
 Let $(M^n,g)$ be a Riemannian manifold and let $(\cmG,\cg)$ be its ambient space.  If $\cL\colon\bigl(\cmE[w]\bigr)^k\to\cmE[w^\prime]$ is a tangential polydifferential operator, then it induces a conformally covariant polydifferential operator $L$ of bidegree $(-w,-w^\prime)$.
\end{lem}

\begin{proof}
 Let $\cu_1,\dotsc,\cu_k\in\cmE[w]$.  Since $\cL$ is tangential, $\cL(\cu_1,\dotsc,\cu_k)$ depends only on $\cu_1\rv_{\mG},\dotsc,\cu_k\rv_{\mG}\in\mE[w]$.  Therefore
 \[ L(u_1,\dotsc,u_k) := \cL(\cu_1,\dotsc,\cu_k)\rv_{\mG} \in \mE[w^\prime] \]
 is well-defined on $\bigl(\mE[w]\bigr)^k$, where $\cu_1,\dotsc,\cu_k\in\cmE[w]$ are such that $\cu_j\rv_{\mG}=u_j$ for all integers $1\leq j\leq k$.  That $L$ is conformally covariant of bidegree $(-w,-w^\prime)$ now follows from~\eqref{eqn:density_transformation}.
\end{proof}

Next we observe that one can use the defining function $Q$ of $\mG$ to identify whether an ambient operator is tangential.

\begin{lem}
 \label{lem:tangential_polydifferential_Q}
 Let $(M^n,g)$ be a Riemannian manifold and let $(\cmG,\cg)$ be its ambient space.  A polydifferential operator $\cL\colon\bigl(\cmE[w]\bigr)^k\to\cmE[w^\prime]$ is tangential if and only if
 \begin{equation}
  \label{eqn:tangential_polydifferential_Q}
  \cL(u_1,\dotsc,u_{j-1},Qz,u_{j+1},\dotsc,u_k) \equiv 0 \mod Q
 \end{equation}
 for all $u_1,\dotsc,u_k\in\cmE[w]$, all $z\in\cmE[w-2]$, and all integers $1\leq j\leq k$.
\end{lem}

\begin{proof}
 It is clear that if $\cL$ is tangential, then~\eqref{eqn:tangential_polydifferential_Q} holds.  Conversely, suppose that~\eqref{eqn:tangential_polydifferential_Q} holds.  Let $u_j^{(1)},u_j^{(2)}\in\cmE[w]$, $1\leq j\leq k$, be such that $u_j^{(1)}\rv_\mG=u_j^{(2)}\rv_\mG$ for all integers $1\leq j\leq k$.  Since $Q$ is a defining function for $\mG$, there are $z_1,\dotsc,z_k\in\cmE[w-2]$ such that $u_j^{(1)}-u_j^{(2)}=Qz_j$ for all integers $1\leq j\leq k$.  Since $\cL$ is polydifferential, we conclude from~\eqref{eqn:tangential_polydifferential_Q} that
 \[ \cL\bigl(u_1^{(1)},\dotsc,u_k^{(1)}\bigr)\rv_\mG = \cL\bigl(u_1^{(2)},\dotsc,u_k^{(2)}\bigr)\rv_{\mG} . \qedhere \]
\end{proof}

We can now study the operator $L_4$ in Theorem~\ref{thm:vk_operator_flat} as an ambient polydifferential operator.  Our computations are simplified by first computing with respect to a non-symmetric polydifferential operator and then considering its symmetrization.  Here, given spaces $V,W$ and a function $F\colon V^k\to W$, $k\in\bN$, we define the \emph{symmetrization $\Sym F\colon V^k\to W$ of $F$} by
\begin{equation}
 \label{eqn:symmetrization}
 \Sym F(v_1,\dotsc,v_k) := \frac{1}{k!} \sum_{\sigma\in S_k} F(v_{\sigma(1)},\dotsc,v_{\sigma(k)})
\end{equation}
for all $v_1,\dotsc,v_k\in V$, where $S_k$ is the symmetric group on $k$ elements.

\begin{thm}
 \label{thm:sigma2_operator}
 Let $(M^n,g)$ be a Riemannian manifold and let $(\cmG,\cg)$ be its ambient space.  Define $\cL_4\colon\bigl(\cmE\bigl[-\frac{n-4}{4}\bigr]\bigr)^3\to\cmE\bigl[-\frac{3n+4}{4}\bigr]$ as the polarization of
 \[ \cL_4(u) = \frac{1}{2}\cdelta\left(\lv\cnabla u\rv^2\,\cd u\right) - \frac{n-4}{16}\left[ u\cDelta\lv\cnabla u\rv^2 - \cdelta\left((\cDelta u^2)\,\cd u\right)\right] . \]
 Then $\cL_4$ is tangential.  Moreover, the operator $L_4$ associated to $\cL_4$ is such that
 \begin{equation}
  \label{eqn:L4_constant}
  L_4(1,1,1) = \left(\frac{n-4}{4}\right)^3\sigma_2^g .
 \end{equation}
 In particular, $L_4$ is associated to the $\sigma_2$-curvature.
\end{thm}

\begin{proof}
 Define $\cL\colon\bigl(\cmE\bigl[-\frac{n-4}{4}\bigr]\bigr)^3\to\cmE\bigl[-\frac{3n+4}{4}\bigr]$ by
 \[ \cL(u,v,w) = \frac{1}{2}\cdelta\left(\lp\cnabla u,\cnabla v\rp\,\cd w\right) - \frac{n-4}{16}\left[ w\cDelta\lp\cnabla u,\cnabla v\rp - \cdelta\left(\cDelta(uv)\,\cd w\right) \right] . \]
 Let $u,v\in\cmE\bigl[-\frac{n-4}{4}\bigr]$ and $z\in\cmE\bigl[-\frac{n+4}{4}\bigr]$.  Since $\lp\cnabla u,\cnabla v\rp,\cDelta(uv)\in\cmE\bigl[-\frac{n}{2}\bigr]$, we deduce that
 \[ \cL(u,v,Qz) \equiv 0 \mod Q . \]
 In particular, $\cL$ is tangential in its last component.  Using~\eqref{eqn:DeltakX} and our weight assumptions, we also compute that
 \[ \cL(Qz,u,v) \equiv 0 \mod Q . \]
 In particular, $\cL$ is tangential in its first component.  Since $\cL$ is symmetric in its first two components, we conclude that $\cL$ it tangential.  Hence $\cL_4=\Sym\cL$ is tangential.
 
 We now verify~\eqref{eqn:L4_constant}.  Since $\cL_4$ is tangential on $\bigl(\cmE\bigl[-\frac{n-4}{4}\bigr]\bigr)^3$, it holds that
 \[ L_4(1,1,1) = \left.\cL_4\left(t^{-\frac{n-4}{4}},t^{-\frac{n-4}{4}},t^{-\frac{n-4}{4}}\right)\right|_{t=1,\rho=0} . \]
 Combining this with~\eqref{eqn:gradient} and~\eqref{eqn:divergence} yields
 \[ L_4(1,1,1) = \left.\frac{n-4}{16}\cdelta\left((\cDelta t^{-\frac{n-4}{2}})\,\cd t^{-\frac{n-4}{4}}\right)\right|_{t=1,\rho=0} = \left(\frac{n-4}{4}\right)^3\sigma_2 . \qedhere \]
\end{proof}

\begin{remark}
 One can also use Theorem~\ref{thm:sigma2_operator} to derive the formula~\eqref{eqn:sigma2_operator} for the formally self-adjoint conformally covariant operator associated to the $\sigma_2$-curvature; see~\cite[Remark~2.2]{Case2019fl}.
\end{remark}

We similarly study the operator $L_6$ in Theorem~\ref{thm:vk_operator_flat} as an ambient polydifferential operator.

\begin{thm}
 \label{thm:v3_operator}
 Let $(M^n,g)$, $n\not=4$, be a Riemannian manifold and let $(\cmG,\cg)$ be its ambient space.  Define $\cL_6\colon\bigl(\cmE\bigl[-\frac{n-6}{6}\bigr]\bigr)^5\to\cmE\bigl[-\frac{5n+6}{6}\bigr]$ as the polarization of
 \begin{multline*}
  \cL_6(u) = -\frac{n^2-3n+18}{72}\cdelta\left(\lv\cnabla u\rv^4\,\cd u\right) \\ + \left(\frac{n-6}{6}\right)\frac{n+3}{36}\left[ \frac{1}{2}u\cDelta\lv\cnabla u\rv^4 - \cdelta\left(\lv\cnabla u\rv^2(\cDelta u^2)\,\cd u\right) \right] \\ + \frac{1}{12}\left(\frac{n-6}{6}\right)^2\left[ u\cdelta\left(T_1(\cnabla^2u^2)(\cnabla\lv\cnabla u\rv^2)\right) - \cdelta\left(\sigma_2(\cnabla^2u^2)\,\cd u\right) \right] . 
 \end{multline*}
 Then $\cL_6$ is tangential.  Moreover, the operator $L_6$ associated to $\cL_6$ is such that
 \begin{equation}
  \label{eqn:L6_constant}
  L_6(1,1,1,1,1) = \left(\frac{n-6}{6}\right)^5v_3 .
 \end{equation}
 In particular, $L_6$ is associated to the renormalized volume coefficient $v_3$.
\end{thm}

\begin{proof}
 Note that the polarization $\cL_6$ of $\cL_6(u)$ is tangential if and only if
 \begin{equation}
  \label{eqn:L6_simplification}
  \cL_6(Qz,u,u,u,u) \equiv 0 \mod Q
 \end{equation}
 for all $u\in\cmE\bigl[-\frac{n-6}{6}\bigr]$ and all $z\in\cmE\bigl[-\frac{n+6}{6}\bigr]$.  Indeed, if $\cL_6$ is tangential, then~\eqref{eqn:L6_simplification} trivially holds.  Conversely, if~\eqref{eqn:L6_simplification} holds, then
 \[ 0 \equiv \cL_6\left(Qz,\sum_{i=1}^4 s_iu_i,\sum_{i=1}^4 s_iu_i,\sum_{i=1}^4 s_iu_i,\sum_{i=1}^4 s_iu_i\right) \mod Q \]
 for all $u_1,\dotsc,u_4\in\cmE\bigl[-\frac{n-6}{6}\bigr]$, all $z\in\cmE\bigl[-\frac{n+6}{6}\bigr]$, and all $s_1,\dotsc,s_4\in\bR$.  Computing the mixed partial derivative $\partial_{s_1}\dotsm\partial_{s_4}$ at $s_1=\dotsm=s_4=0$ and using the symmetry of $\cL_6$ yields
 \[ \cL_6(Qz,u_1,u_2,u_3,u_4) \equiv 0 \mod Q \]
 for all $u_1,\dotsc,u_4\in\cmE\bigl[-\frac{n-6}{6}\bigr]$ and all $z\in\cmE\bigl[-\frac{n+6}{6}\bigr]$.  It follows that $\cL_6$ is tangential.
 
 Define $\cL_6^j\colon\bigl(\cmE\bigl[-\frac{n-6}{6}\bigr]\bigr)^5\to\cmE\bigl[-\frac{5n+6}{6}\bigr]$, $j=1,2,3$, by
 \begin{align*}
  \cL_6^1(u_1,u_2,v_1,v_2,w) & := \cdelta\left(\lp\cnabla u_1,\cnabla u_2\rp\lp\cnabla v_1,\cnabla v_2\rp\,\cd w\right), \\
  \cL_6^2(u_1,u_2,v_1,v_2,w) & := \frac{1}{2}w\cDelta\left(\lp\cnabla u_1,\cnabla u_2\rp\lp\cnabla v_1,\cnabla v_2\rp\right) \\
   & \qquad - \cdelta\left(\lp\cnabla u_1,\cnabla u_2\rp\cDelta(v_1v_2)\,\cd w\right) , \\
  \cL_6^3(u_1,u_2,v_1,v_2,w) & := w\cdelta\left(T_1(\cnabla^2(v_1v_2)(\cnabla\lp\cnabla u_1,\cnabla u_2\rp)\right) \\
   & \qquad - \cdelta\left(\sigma_2\bigl(\cnabla^2(u_1u_2),\cnabla^2(v_1v_2)\bigr)\,\cd w\right) .
 \end{align*}
 Note the symmetries
 \[ \cL_6^j(u_1,u_2,v_1,v_2,w)=\cL_6^j(u_2,u_1,v_1,v_2,w)=\cL_6^j(u_1,u_2,v_2,v_1,w) \]
 for all $u_1,u_2,v_1,v_2,w\in\cmE\bigl[-\frac{n-6}{6}\bigr]$ and all $j=1,2,3$.  Note also that if $z\in\cmE\bigl[-\frac{n+6}{6}\bigr]$, then $Qz\in\cmE\bigl[-\frac{n-6}{6}\bigr]$.  We now perform some computations.
 
 First, if $u\in\cmE\bigl[-\frac{n-6}{6}\bigr]$, then each of $\lv\cnabla u\rv^4$, $\lv\cnabla u\rv^2\cDelta u^2$ and $\sigma_2(\cnabla^2u^2)$ is in $\cmE\bigl[-\frac{2n}{3}\bigr]$.  It readily follows that
 \begin{equation}
  \label{eqn:v3_operator_5}
  \cL_6^j(u,u,u,u,Qz) \equiv 0 \mod Q
 \end{equation}
 for all $j=1,2,3$, all $u\in\cmE\bigl[-\frac{n-6}{6}\bigr]$, and all $z\in\cmE\bigl[-\frac{n+6}{6}\bigr]$.
 
 Second, our weight assumptions imply that
 \begin{multline}
  \label{eqn:v3_operator0}
  \cL_6^1(Qz,u,u,u,u) + \cL_6^1(u,u,Qz,u,u) \\ \equiv -\frac{2(n-6)}{3}\left[ z\cdelta\left(u\lv\cnabla u\rv^2\,\cd u\right) + 2u\lv\cnabla u\rv^2\lp\cnabla z,\cnabla u\rp \right] \mod Q
 \end{multline}
 for all $u\in\cmE\bigl[-\frac{n-6}{6}\bigr]$ and all $z\in\cmE\bigl[-\frac{n+6}{6}\bigr]$.
 
 Third, Equation~\eqref{eqn:DeltakX} implies that
 \begin{multline}
  \label{eqn:third1} \cDelta\left(\lv\cnabla u\rv^2\lp\cnabla(Qz),\cnabla u\rp\right)\equiv -\frac{n-6}{3}\cDelta\left(uz\lv\cnabla u\rv^2\right) \\ - \frac{2(n+6)}{3}\lv\cnabla u\rv^2\lp\cnabla z,\cnabla u\rp \mod Q
 \end{multline}
 for all $u\in\cmE\bigl[-\frac{n-6}{6}\bigr]$ and all $z\in\cmE\bigl[-\frac{n+6}{6}\bigr]$.  Using~\eqref{eqn:DeltakX} and~\eqref{eqn:third1}, we deduce that
 \begin{align*}
  \cL_6^2(Qz,u,u,u,u) & \equiv -\frac{n+6}{3}u\lv\cnabla u\rv^2\lp\cnabla z,\cnabla u\rp - \frac{n-6}{6}u\cDelta(zu\lv\cnabla u\rv^2) \\
   & \quad + \frac{n-6}{3}\cdelta\left(zu(\cDelta u^2)\,\cd u\right) + \frac{n-6}{3}u\lp\cnabla z,\cnabla u\rp\cDelta u^2 \mod Q , \\
  \cL_6^2(u,u,Qz,u,u) & \equiv -\frac{n+6}{3}u\lv\cnabla u\rv^2\lp\cnabla z,\cnabla u\rp - \frac{2(n+6)}{3}\cdelta\left(zu\lv\cnabla u\rv^2\,\cd u\right) \\
   & \quad - \frac{n-6}{6}u\cDelta(zu\lv\cnabla u\rv^2) + \frac{n-6}{3}u\lv\cnabla u\rv^2\cDelta(zu) \mod Q
 \end{align*}
 for all $u\in\cmE\bigl[-\frac{n-6}{6}\bigr]$ and all $z\in\cmE\bigl[-\frac{n+6}{6}\bigr]$.  In particular,
 \begin{equation}
  \label{eqn:v3_operator1}
  \begin{split}
   \MoveEqLeft \cL_6^2(Qz,u,u,u,u) + \cL_6^2(u,u,Qz,u,u) \\
   & \equiv -\frac{2(n+6)}{3}\left(z\cdelta\left(u\lv\cnabla u\rv^2\,\cd u\right) + 2u\lv\cnabla u\rv^2\lp\cnabla z,\cnabla u\rp\right)\\
    & \quad - \frac{n-6}{3}\left( u^2z\cDelta\lv\cnabla u\rv^2 - z\cdelta\left(u(\cDelta u^2)\,\cd u\right) \right) \\
    & \quad - \frac{2(n-6)}{3}u\lp\cnabla(uz),\cnabla\lv\cnabla u\rv^2\rp + \frac{2(n-6)}{3}u\lp\cnabla z,\cnabla u\rp\,\cDelta u^2 \mod Q
  \end{split}
 \end{equation}
 for all $u\in\cmE\bigl[-\frac{n-6}{6}\bigr]$ and all $z\in\cmE\bigl[-\frac{n+6}{6}\bigr]$.
 
 Fourth, if $f\in\cmE\bigl[-\frac{n}{3}\bigr]$, then
 \begin{equation}
  \label{eqn:T1}
  T_1\bigl(\cnabla^2(Qf)\bigr) = QT_1(\cnabla^2f) - 2X\otimes\cnabla f - 2(\cnabla f)\otimes X + \frac{2(n+3)}{3}f\cg .
 \end{equation}
 Since $2\sigma_2(\cnabla^2f_1,\cnabla^2f_2)=\lp T_1(\cnabla^2f_1),\cnabla^2f_2\rp$, we deduce that
 \begin{equation}
  \label{eqn:sigma2}
  \sigma_2(\cnabla^2(Qf_1),\cnabla^2f_2) = Q\sigma_2(\cnabla^2f_1,\cnabla^2f_2) + \frac{2(n-3)}{3}\lp\cnabla f_1,\cnabla f_2\rp + \frac{n+3}{3}f_1\cDelta f_2
 \end{equation}
 for all $f_1\in\cmE\bigl[-\frac{n}{3}\bigr]$ and all $f_2\in\cmE\bigl[-\frac{n-6}{3}\bigr]$.  Using~\eqref{eqn:sigma2}, we deduce that
 \begin{multline}
  \label{eqn:sigma2_Q}
  \delta\left(\sigma_2(\cnabla^2(Quz),\cnabla^2u^2)\,\cd u\right) \equiv -\frac{n-6}{3}u\sigma_2(\cnabla^2(uz),\cnabla^2u^2) \\ + \frac{n+3}{3}\cdelta\left(uz(\cDelta u^2)\,\cd u\right) + \frac{2(n-3)}{3}\cdelta\left(\lp\cnabla(uz),\cnabla u^2\rp\,\cd u\right) \mod Q
 \end{multline}
 for all $u\in\cmE\bigl[-\frac{n-6}{6}\bigr]$ and all $z\in\cmE\bigl[-\frac{n+6}{6}\bigr]$.  Using~\eqref{eqn:T1} and the fact that the divergence of $T_1(\cnabla^2f)$ is zero along $\mG$ in $(\cmG,\cg)$ for any $f\in C^\infty(\cmG)$, we deduce that
 \begin{multline}
  \label{eqn:T1-firstQ}
  u\cdelta\left(T_1(\cnabla^2(Quz))(\cnabla\lv\cnabla u\rv^2)\right) \equiv \frac{4(n+3)}{3}u\lp\cnabla(uz),\cnabla\lv\cnabla u\rv^2\rp \\ + \frac{2(n+3)}{3}u^2z\cDelta\lv\cnabla u\rv^2 \mod Q
 \end{multline}
 for all $u\in\cmE\bigl[-\frac{n-6}{6}\bigr]$ and all $z\in\cmE\bigl[-\frac{n+6}{6}\bigr]$.  Using these two facts again yields
 \begin{multline}
  \label{eqn:T1-secondQ}
  u\cdelta\left(T_1(\cnabla^2u^2)(\cnabla\lp\cnabla(Qz),\cnabla u\rp)\right) \equiv -\frac{2(n-6)}{3}u\sigma_2(\cnabla^2u^2,\cnabla^2(uz)) \\ + \frac{2(n-9)}{3}u\lp\cnabla z,\cnabla u\rp\,\cDelta u^2 + \frac{4(n-3)}{3}u\lp\cnabla\lp\cnabla z,\cnabla u\rp,\cnabla u^2\rp \mod Q
 \end{multline}
 for all $u\in\cmE\bigl[-\frac{n-6}{6}\bigr]$ and all $z\in\cmE\bigl[-\frac{n+6}{6}\bigr]$.  Combining~\eqref{eqn:sigma2_Q}, \eqref{eqn:T1-firstQ} and~\eqref{eqn:T1-secondQ} yields
 \begin{equation}
  \label{eqn:v3_operator2}
  \begin{split}
   \MoveEqLeft \cL_6^3(Qz,u,u,u,u) + \cL_6^3(u,u,Qz,u,u) \\
   & \equiv -\frac{8(n-3)}{3}\left(z\cdelta\left(u\lv\cnabla u\rv^2\,\cd u\right) + 2u\lv\cnabla u\rv^2\lp\cnabla u,\cnabla z\rp\right) \\
   & \quad + \frac{2(n+3)}{3}\left( u^2z\cDelta\lv\cnabla u\rv^2 - z\cdelta\left(u(\cDelta u^2)\,\cd u\right) \right) \\
   & \quad + \frac{4(n+3)}{3}u\lp\cnabla(uz),\cnabla\lv\cnabla u\rv^2\rp - \frac{4(n+3)}{3}u\lp\cnabla u,\cnabla z\rp\,\cDelta u^2 \mod Q
  \end{split}
 \end{equation}
 for all $u\in\cmE\bigl[-\frac{n-6}{6}\bigr]$ and all $z\in\cmE\bigl[-\frac{n+6}{6}\bigr]$.
 
 Finally, combining~\eqref{eqn:v3_operator0}, \eqref{eqn:v3_operator1} and~\eqref{eqn:v3_operator2} implies that~\eqref{eqn:L6_simplification} holds.  Therefore the polarization $\cL_6$ of $\cL_6(u)$ is tangential.
 
 We now verify~\eqref{eqn:L6_constant}.  Since $\cL_6$ is tangential on $\bigl(\cmE\bigl[-\frac{n-6}{6}\bigr]\bigr)^5$, it holds that
 \[ L_6(1,1,1,1,1) = \left.\cL_6\left(t^{-\frac{n-6}{6}},t^{-\frac{n-6}{6}},t^{-\frac{n-6}{6}},t^{-\frac{n-6}{6}},t^{-\frac{n-6}{6}}\right)\right|_{t=0,\rho=1} . \]
 Combining this with~\eqref{eqn:gradient}, \eqref{eqn:divergence} and~\eqref{eqn:hessian} yields
 \[ L_6(1,1,1,1,1) = \frac{1}{3}\left(\frac{n-6}{6}\right)^5\left.\left(\partial_\rho + J\right)\right|_{\rho=0}\sigma_2(U) \]
 for $U$ as in~\eqref{eqn:U}.  The conclusion follows from the identity
 \[ \sigma_2(U) = \sigma_2(P) - \rho\left(J\lv P\rv^2 - Y\right) + o(\rho) . \qedhere \]
\end{proof}

The verifications of~\eqref{eqn:L4_constant} and~\eqref{eqn:L6_constant} show that, even if the operator $\cL_{2k}$, $k\geq4$, defined by~\eqref{eqn:vk_operator_flat} is tangential, it does not induce an operator associated to the renormalized volume coefficient $v_k$.  To see this, note that if $\cL_{2k}$ is tangential, then~\eqref{eqn:gradient}, \eqref{eqn:divergence}, \eqref{eqn:hessian} imply that the operator $L_{2k}$ induced by $\cL_{2k}$ satisfies
\[ L_{2k}(1,\dotsc,1) = \frac{1}{k}\left(\frac{n-2k}{2k}\right)^{2k-1}\left[ J\sigma_{k-1}(P) - \left\lp T_{k-2}(P), P^2 + \frac{1}{n-4}B\right\rp \right] . \]
This expression does not involve the higher-order extended obstruction tensors, and hence cannot agree with a multiple of the renormalized volume coefficient $v_k$~\cite{Graham2009}.

%% file: trilinear.tex
\section{Two sixth-order operators of rank four}
\label{sec:trilinear}

As pointed out in our previous work~\cite{CaseLinYuan2016}, the space of CVIs of weight $-6$ is ten-dimensional.  The subspace of such CVIs which are trivial on locally conformally flat manifolds is six-dimensional.  A basis for a four-dimensional complementary subspace was identified in our previous work~\cite{CaseLinYuan2016}.  One element of our basis is Branson's sixth-order $Q$-curvature~\cite{Branson1995}.  Up to a multiplicative constant, this is the only possible basis element of rank $2$ which is nontrivial on locally conformally flat manifolds.  Another element of our basis is the third renormalized volume coefficient, which has rank $6$.  By Lemma~\ref{lem:kernel-pi}, we thus expect we can modify our previous basis so that the remaining two basis elements have rank four.  This is claimed by Theorem~\ref{thm:rank4} and proven in this section.

We prove Theorem~\ref{thm:rank4} by first constructing sixth-order conformally covariant tridifferential operators using the ambient metric, and then identifying their constant terms as CVIs by comparing them to our previous basis~\cite{CaseLinYuan2016}.  As in Section~\ref{sec:ambient}, our formulae are greatly simplified by working with non-symmetric elements of $\Poly$ and only symmetrizing at that end.

The first step in our proof of Theorem~\ref{thm:rank4} is to find two tangential sixth-order tridifferential operators on the ambient manifold.  This is accomplished in the following two propositions.

\begin{prop}
 \label{prop:operatorB}
 Let $(M^n,g)$ be a Riemannian manifold and let $(\cmG,\cg)$ be its ambient manifold.  Define $\cB\colon\bigl(\cmE\bigl[-\frac{n-6}{4}\bigr]\bigr)^{3}\to\cmE\bigl[-\frac{3n+6}{4}\bigr]$ by
 \[ \cB(\cu,\cv,\cw) = \cdelta\left(\cDelta(\lp\cnabla\cu,\cnabla\cv\rp)\,\cd\cw\right) - \frac{n-6}{16}\left[ \cw\cDelta^2(\lp\cnabla\cu,\cnabla\cv\rp) - \cdelta\left(\cDelta^2(\cu\cv)\,\cd\cw\right) \right] \]
 Then $\cB$ is tangential.  In particular, the operator $B$ induced by $\Sym\cB$ is a conformally covariant tridifferential operator of bidegree $\left(\frac{n-6}{4},\frac{3n+6}{4}\right)$.
\end{prop}

\begin{proof}
 First observe that $\cB(\cu,\cv,\cw)=\cB(\cv,\cu,\cw)$ for all $\cu,\cv,\cw\in\cmE\bigl[-\frac{n-6}{4}\bigr]$, so it suffices to show that $\cB$ is tangential in its first and third components.
 
 We first consider $\cB(\cu,\cv,Q\cw)$ for $\cu,\cv,Q\cw\in\cmE\bigl[-\frac{n-6}{4}\bigr]$.  Hence $\cw\in\cmE\bigl[-\frac{n+2}{4}\bigr]$.  Note that $\cDelta\lp\cnabla\cu,\cnabla\cv\rp,\cDelta^2(\cu\cv)\in\cmE\bigl[-\frac{n+2}{2}\bigr]$.  Let $\cf\in\cmE\bigl[-\frac{n+2}{2}\bigr]$.  It follows from~\eqref{eqn:divf} that $\cdelta\bigl(\cf\,\cd (Q\cw)\bigr)=O(Q)$, and hence $\cw\mapsto\cdelta\bigl(\cf\,\cd\cw\bigr)$ is tangential on $\cmE\bigl[-\frac{n-6}{4}\bigr]$.  Therefore $\cB$ is tangential in its third component.
 
 We now consider $\cB(Q\cu,\cv,\cw)$ for $\cv,\cw\in\cmE\bigl[-\frac{n-6}{4}\bigr]$ and $\cu\in\cmE\bigl[-\frac{n+2}{4}\bigr]$.  Using~\eqref{eqn:DeltakX} we compute that
 \begin{align*}
  \cDelta^2\lp\cnabla(Q\cu),\cnabla\cv\rp & = -\frac{n-6}{2}\cDelta^2(\cu\cv) - 8\cDelta\lp\cnabla\cu,\cnabla\cv\rp + Q\cDelta^2\lp\cnabla\cu,\cnabla\cv\rp , \\
  \cDelta^2(Q\cu\cv) & = 8\cDelta(\cu\cv) + Q\cDelta^2(\cu\cv) .
 \end{align*}
 Combining this with~\eqref{eqn:divQ} yields
 \begin{align*}
  \cw\cDelta^2\left(\lp\cnabla (Q\cu),\cnabla\cv\rp\right) - \cdelta\left(\cDelta^2(\cu\cv)\,\cd\cw\right) & = -8\cw\cDelta\lp\cnabla\cu,\cnabla\cv\rp - 8\cdelta\left(\cDelta(\cu\cv)\,\cd\cw\right) \\
   & \quad + Q\left(\cw\cDelta^2\lp\cnabla\cu,\cnabla\cv\rp - \cdelta\left(\cDelta^2(\cu\cv)\,\cd\cw\right)\right) , \\
  \cdelta\left(\cDelta\left(\lp\cnabla(Q\cu),\cnabla\cv\rp\right)\,\cd\cw\right) & = -\frac{n-6}{2}\cdelta\left(\cDelta(\cu\cv)\,\cd\cw\right) - \frac{n-6}{2}\cw\,\cDelta\lp\cnabla\cu,\cnabla\cv\rp \\
   & \quad + Q\cdelta\left(\cDelta\left(\lp\cnabla\cu,\cnabla\cv\rp\right)\,\cd\cw\right) .
 \end{align*}
 It follows that $\cB$ is tangential in its first component.
 
 Finally, since $\cB$ is tangential, its symmetrization $\Sym\cB$ is also tangential.  The final conclusion follows from Lemma~\ref{lem:tangential_polydifferential}.
\end{proof}

\begin{prop}
 \label{prop:operatorC}
 Let $(M^n,g)$ be a Riemannian manifold and let $(\cmG,\cg)$ be its ambient manifold.  Define $\cC\colon\bigl(\mE\bigl[-\frac{n-6}{4}\bigr]\bigr)^{3}\to\mE\bigl[-\frac{3n+6}{4}\bigr]$ by
 \begin{align*}
  \cC & = \Sym\left(\cH - \cF - \frac{2(n+2)}{3(n-2)}\cG - \frac{n-6}{6(n+2)}\cK\right) ,
 \end{align*}
 where $\Sym$ denotes symmetrization,
 \begin{align*}
  \cH(\cu,\cv,\cw) & := \cdelta\left((\cDelta u)(\cDelta v)\,\cd\cw\right) , \\
  \cF(\cu,\cv,\cw) & := \cDelta\left(\lp\cnabla u,\cnabla v\rp\,\cDelta w\right) , \\
  \cG(\cu,\cv,\cw) & := \cnabla^\alpha\left(\delta_{\alpha\beta\gamma}^{\lambda\mu\nu}\cw_\lambda\cu_\mu^\beta\cv_\nu^\gamma\right) , \\
  \cK(\cu,\cv,\cw) & := (\cDelta\cu)(\cDelta\cv)(\cDelta\cw) + 3\cDelta\left(\cw(\cDelta\cu)(\cDelta\cv)\right),
 \end{align*}
 and $\delta_{\alpha\beta\gamma}^{\lambda\mu\nu}$ is the generalized Kronecker symbol.  Then $\cC$ is tangential.  In particular, the operator $C$ induced by $\cC$ is a conformally covariant tridifferential operator of bidegree $\left(\frac{n-6}{4},\frac{3n+6}{4}\right)$.
\end{prop}

\begin{proof}
 Observe that each of $\cH,\cF,\cG,\cK$ is symmetric in its first two components.  Given $\cu,\cv\in\cmE\bigl[-\frac{n-6}{4}\bigr]$ and $\cz\in\cmE\bigl[-\frac{n+2}{4}\bigr]$, it thus suffices to compute $\cH(\cu,\cv,Q\cz)$ and $\cH(Q\cz,\cu,\cv)$ modulo $Q$, and similarly for $\cF$, $\cG$ and $\cK$.  Throughout this proof we write $A\equiv B$ to denote that $A-B=O(Q)$.
 
 First consider $\cH$.  A straightforward computation gives
 \begin{align*}
  \cH(\cu,\cv,Q\cz) & \equiv 0 , \\
  \cH(Q\cz,\cu,\cv) & \equiv (n+2)\cdelta\left(\cz(\cDelta\cu)\,\cd\cv\right) - \frac{n-6}{2}\cv(\cDelta\cu)(\cDelta\cz) .
 \end{align*}
 
 Second consider $\cF$.  A straightforward computation gives
 \begin{align*}
  \cF(\cu,\cv,Q\cz) & \equiv (n+2)\left[\cDelta\left(\cz\lp\cnabla\cu,\cnabla\cv\rp\right) - \lp\cnabla\cu,\cnabla\cv\rp\,\cDelta\cz\right] \\
  \cF(Q\cz,\cu,\cv) & \equiv -(n+2)\lp\cnabla\cu,\cnabla\cz\rp\,\cDelta\cv - \frac{n-6}{2}\cDelta\left(\cu\cz\cDelta\cv\right) .
 \end{align*}
 
 Third consider $\cG$.  Since the ambient metric is Ricci flat, we compute that
 \[ \cG(\cu,\cv,\cw) = \delta_{\alpha\beta\gamma}^{\lambda\mu\nu}\cu_\lambda^\alpha\cv_\mu^\beta\cw_\nu^\gamma - \frac{1}{2}R^{\alpha\beta}{}_\gamma{}^\rho\delta_{\alpha\beta}^{\mu\nu}w_\mu(u_\rho v_\nu^\gamma + u_\nu^\gamma v_\rho) . \]
 Using the facts~\eqref{eqn:Xcurv} and
 \[ (Q\cz)_{\alpha\beta} \equiv 2X_\alpha \cz_\beta + 2\cz_\alpha X_\beta + 2\cz\cg_{\alpha\beta}, \]
 we conclude that
 \begin{multline*}
  \cG(\cu,\cv,Q\cz) \equiv \cG(Q\cz,\cu,\cv) \equiv \delta_{\alpha\beta\gamma}^{\lambda\mu\nu}\cu_\lambda^\alpha\cv_\mu^\beta(Q\cz)_\nu^\gamma \\ \equiv (n-2)\left[ 2\cz\sigma_2(\cnabla^2\cu,\cnabla^2\cv) + T_1(\cnabla^2\cu)(\cnabla\cv,\cnabla\cz) + T_1(\cnabla^2\cv)(\cnabla\cu,\cnabla\cz) \right] .
 \end{multline*}

 Fourth consider $\cK$.  A straightforward computation gives
 \begin{align*}
  \cK(\cu,\cv,Q\cz) & \equiv -2(n+2)\cz(\cDelta\cu)(\cDelta\cv) , \\
  \cK(Q\cz,\cu,\cv) & \equiv (n+2)\left[\cz(\cDelta\cu)(\cDelta\cv) + 3\cDelta(\cv\cz\cDelta\cu) - 3\cv(\cDelta\cz)(\cDelta\cu)\right] .
 \end{align*}

 Combining these formulae yields the desired result.
\end{proof}

The second step in our proof of Theorem~\ref{thm:rank4} is to identify the constant term in the restriction of the operators produced above to the underlying manifold.

\begin{prop}
 \label{prop:operator_scalars}
 Let $(M^n,g)$ be a Riemannian manifold and let $B$ and $C$ be the operators defined in Proposition~\ref{prop:operatorB} and Proposition~\ref{prop:operatorC}, respectively.  Then
 \begin{align*}
  B(1,1,1) & = \left(\frac{n-6}{4}\right)^3B_0 , \\
  C(1,1,1) & = \left(\frac{n-6}{4}\right)^3C_0 , 
 \end{align*}
 where
 \begin{align*}
  B_0 & := -\frac{3}{4}\Delta J^2 + \Delta\sigma_2 + \delta\left(T_1(\nabla J)\right) - \frac{n-6}{4}J^3 + \frac{n-6}{2}J\lv P\rv^2 - 6v_3, \\
  C_0 & := -\frac{2}{n+2}\Delta J^2 + \frac{2(n-6)}{3(n+2)}J^3 + \frac{4(n+2)}{n-2}v_3 .
 \end{align*}
\end{prop}

\begin{proof}
 It follows from Lemma~\ref{lem:local_ambient_formulas} and the definition of $\cG$ from Proposition~\ref{prop:operatorC} that \[ \cG(t^w,t^w,t^w) = 2w^3t^{3(w-2)}(\partial_\rho+J)\sigma_2(U) + O(\rho) . \]Therefore
 \begin{equation}
  \label{eqn:cKt}
  \cG(t^w,t^w,t^w) = 6w^3t^{3(w-2)}v_3 + O(\rho) .
 \end{equation}
 
 We first apply Lemma~\ref{lem:local_ambient_formulas} to the operator $B$ defined by Proposition~\ref{prop:operatorB}.  Since $\cB$ is tangential, it holds that
 \[ B(1,1,1) = \left.\cB(t^w,t^w,t^w)\right|_{\rho=0,t=1} \]
 for $w=-\frac{n-6}{4}$.  Hence, by~\eqref{eqn:gradient} and~\eqref{eqn:divergence},
 \begin{equation}
  \label{eqn:B111}
  B(1,1,1) = -\frac{w^2}{4}\left.\left(\partial_\rho + J\right)\right|_{\rho=0,t=1}\cDelta^2t^{2w} .
 \end{equation}
 Applying~\eqref{eqn:divergence} again yields
 \[ \cDelta t^{2w} = 2wt^{2w-2}\left( J - \rho\lv P\rv^2 + \rho^2Y\right) + O(\rho^3) . \]
 Taking the ambient Laplacian of both sides via~\eqref{eqn:Laplacian} then yields
 \begin{multline*}
  \cDelta^2t^{2w} = 2wt^{2w-4}\biggl[ \Delta J - \frac{n-2}{2}J^2 \\ + \rho\Bigl(-\frac{3}{2}\Delta J^2 + 2\Delta\sigma_2 + 2\delta\left(T_1(\nabla J)\right) - J\Delta J - 4Y + nJ\lv P\rv^2\Bigr) \biggr] + O(\rho^2) .
 \end{multline*}
 Inserting this into~\eqref{eqn:B111} yields the desired formula for $B(1,1,1)$.
 
 We conclude by applying Lemma~\ref{lem:local_ambient_formulas} to the operator $C$ defined by Proposition~\ref{prop:operatorC}.  Since $\cC$ is tangential, it holds that
 \[ C(1,1,1) = \left.\cC(t^w,t^w,t^w)\right|_{\rho=0,t=1} \]
 for $w=-\frac{n-6}{4}$.  Hence, by~\eqref{eqn:gradient},
 \begin{multline}
  \label{eqn:C111}
  C(1,1,1) = \cdelta\left((\cDelta t^w)^2\,\cd t^w\right) \\ + \frac{2w}{3(n+2)}\left(\bigl(\cDelta t^w\bigr)^3 + 3\cDelta\bigl(t^w(\cDelta t^w)^2\bigr)\right) - \frac{2(n+2)}{3(n-2)}\cG(t^w,t^w,t^w) .
 \end{multline}
 We deduce from~\eqref{eqn:divergence} that
 \begin{align*}
  \bigl(\cDelta t^w\bigr)^2 & = w^2t^{2(w-2)}\left(J^2-2\rho J\lv P\rv^2\right) + O(\rho^2), \\
  \bigl(\cDelta t^w\bigr)^3 & = w^3t^{3(w-2)}J^3 + O(\rho) .
 \end{align*}
 Combining these with~\eqref{eqn:divergence} and~\eqref{eqn:Laplacian} yields
 \begin{align*}
  \left.\delta\bigl(\bigl(\cDelta t^w\bigr)^2\,\cd t^w\bigr)\right|_{\rho=0,t=1} & = w^3\left(J^3 - 2J\lv P\rv^2\right) , \\
  \left.\cDelta\left(t^w\bigl(\cDelta t^w\bigr)^2\right)\right|_{\rho=0,t=1} & = w^2\left(\Delta J^2 + (n+2)J\lv P\rv^2 - \frac{3n-2}{4}J^3\right) .
 \end{align*}
 Combining these observations with~\eqref{eqn:cKt} and~\eqref{eqn:C111} yields the desired formula for $C(1,1,1)$.
\end{proof}

The proof of Theorem~\ref{thm:rank4} will be complete once we check that $I_1$ and $I_2$ are CVIs and compute their rank.

\begin{proof}[Proof of Theorem~\ref{thm:rank4}]
 The authors previously showed~\cite{CaseLinYuan2016} that
 \begin{align*}
  L_1 & := -\Delta J^2 + \frac{n-6}{3}J^3 , \\
  L_2 & := -\Delta\lv P\rv^2 - 2\delta\left(P(\nabla J)\right) - \Delta J^2 + (n-6)J\lv P\rv^2
 \end{align*}
 and $v_3$ are CVIs of weight $-6$.  Observe that
 \begin{align*}
  B_0 & = -\frac{3}{4}L_1 + \frac{1}{2}L_2 - 6v_3 , \\
  C_0 & = \frac{2}{n+2}L_1 + \frac{4(n+2)}{n-2}v_3 .
 \end{align*}
 and
 \begin{align*}
  I_1 & = \frac{n+2}{2}C_0, \\
  I_2 & = \frac{3(n+2)}{8}C_0 - B_0 .
 \end{align*}
 As linear combinations of CVIs of the same weight are CVIs, we conclude that $B_0$ and $C_0$, and hence $I_1$ and $I_2$, are CVIs.
 
 Next, Proposition~\ref{prop:operator_scalars} implies that $I_1$ and $I_2$ both have rank at most four.  A straightforward computation shows that none of the operators $L_1^1$, $L_2^2$ or $L_3^3$ associated to these invariants vanishes.  Therefore $I_1$ and $I_2$ are both of rank four.
\end{proof}

\begin{remark}
 \label{rk:fsa}
 Combining Theorem~\ref{thm:rank4}, Theorem~\ref{thm:rank_existence} and Proposition~\ref{prop:operator_scalars} implies that the conformally covariant tridifferential operators $B$ and $C$ are formally self-adjoint.  This is not immediately clear from their construction; see the discussion surrounding Conjecture~\ref{conj:ovsienko_redou} below for further discussion of possible alternative proofs of formal self-adjointness. 
\end{remark}

%% file: ovsienko_redou.tex
\section{A family of rank $3$ conformally covariant operators}
\label{sec:ovsienko_redou}

Ovsienko and Redou~\cite{OvsienkoRedou2003} proved that for all $n\in\bN$ and all $k\in\bN_0$, and all but finitely many choices of $(\lambda,\mu)\in\bR^2$, there is a unique (up to multiplicative constant) conformally covariant bidifferential operator $D\colon\mE[\lambda]\times\mE[\mu]\to\mE[\lambda+\mu-2k]$ on the round $n$-sphere.  When $n>2k$, their result applies to the choice $\lambda=\mu=-\frac{n-2k}{3}$ which, by Lemma~\ref{lem:find_weights}, is the only operator in this family which can be formally self-adjoint.  In this section, we use the ambient metric to give a new construction of the Ovsienko--Redou operators with $\lambda=\mu=-\frac{n-2k}{3}$.  As a consequence, we construct curved analogues of these operators on all Riemannian manifolds of dimension $n\geq2k$.  This dimensional restriction ensures that the ambiguity of the ambient metric is not seen in our construction.

\begin{thm}
 \label{thm:ovsienko_redou}
 Let $(M^n,g)$ be a Riemannian manifold with ambient space $(\cmG,\cg)$ and let $k\leq n/2$ be a nonnegative integer.  Given $u,v\in\cmE\bigl[-\frac{n-2k}{3}\bigr]$, define
 \[ \cD_{2k}(u,v) := (-1)^k\sum_{s=0}^k\sum_{t=0}^{k-s} a_{k-s-t,s,t}\cDelta^{k-s-t}\left(\bigl(\cDelta^su\bigr)\bigl(\cDelta^tv\bigr)\right) , \]
 where $\cDelta$ is the Laplacian of $\cg$ and
 \[ a_{r,s,t} = \frac{k!}{r!s!t!}\frac{\Gamma\bigl(\frac{n+4k}{6}-r\bigr)\Gamma\bigl(\frac{n+4k}{6}-s\bigr)\Gamma\bigl(\frac{n+4k}{6}-t\bigr)}{\Gamma\bigl(\frac{n-2k}{6}\bigr)\Gamma\bigl(\frac{n+4k}{6}\bigr)^2} . \]
 for $k:=r+s+t$.  Then $\cD_{2k}$ is tangential, and in particular induces a conformally covariant bidifferential operator $D_{2k}$ of bidegree $\left(\frac{n-2k}{3},\frac{2n+2k}{3}\right)$.
\end{thm}

\begin{remark}
 \label{rk:symmetry}
 Note that $a_{r,s,t}$ is symmetric in $r,s,t$; i.e.
 \[ a_{r,s,t} = a_{r,t,s} = a_{s,r,t} = a_{s,t,r} = a_{t,r,s} = a_{t,s,r} . \] 
\end{remark}

\begin{proof}
 Remark~\ref{rk:symmetry} implies that $\cD_{2k}(u,v)=\cD_{2k}(v,u)$.  It thus suffices to show that $\cD_{2k}(Qu,v)\equiv0\mod Q$ for all $v\in\cmE\bigl[-\frac{n-2k}{3}\bigr]$ and $u\in\cmE\bigl[-\frac{n-2k+6}{3}\bigr]$.  This follows immediately from the identity
 \begin{multline*}
  \cDelta^r\left((\cDelta^s Qu)\cDelta^tv\right) \equiv 2s\left(\frac{n+4k}{3}-2s\right)\cDelta^{k-s-t}\left((\cDelta^{s-1}u)\cDelta^tv\right) \\ - 2r\left(\frac{n+4k}{3} - 2r\right)\cDelta^{r-1}\left((\cDelta^s u)\cDelta^tv\right) \mod Q
 \end{multline*}
 and the definition of $a_{r,s,t}$.  Since $\cD_{2k}$ is tangential, Lemma~\ref{lem:tangential_polydifferential} implies that it induces a conformally covariant operator of bidegree $\left(\frac{n-2k}{3},\frac{2n+2k}{3}\right)$.
\end{proof}

The symmetries of Remark~\ref{rk:symmetry} imply that $\cD_{2k}$ is formally self-adjoint as an operator on $(\cmG,\cg)$.  For this reason we expect that the induced operator $D_{2k}$ is formally self-adjoint on $(M,g)$.

\begin{conj}
 \label{conj:ovsienko_redou}
 Let $(M^n,g)$ be a Riemannian manifold and let $k\leq n/2$ be a positive integer.  Then the operator $D_{2k}$ of Theorem~\ref{thm:ovsienko_redou} is formally self-adjoint.
\end{conj}

Fefferman and Graham~\cite{FeffermanGraham2013} recently gave a direct proof that the GJMS operators are formally self-adjoint using only their ambient construction~\cite{GJMS1992} and a clever combinatorial argument found by Juhl and Krattenthaler~\cite{Juhl2013}.  Alternative proofs of the formal self-adjointness of the GJMS operators involve their relation to scattering theory~\cite{GrahamZworski2003} or their definition through formal properties of Poincar\'e manifolds~\cite{FeffermanGraham2002}.  It seems more likely to us that a generalization of the direct argument of Fefferman--Graham and Juhl--Krattenthaler will verify Conjecture~\ref{conj:ovsienko_redou}.

It is straightforward to derive formulae for the Ovsienko--Redou operators for small $k$, and thereby verify Conjecture~\ref{conj:ovsienko_redou} in these cases.

\begin{thm}
 \label{thm:eval_ovsienko_redou}
 Let $(M^n,g)$ be a Riemannian manifold.  If $n\geq2$, then
 \begin{equation}
  \label{eqn:eval_ovsienko_redou1}
  D_2(u,v) = -\Delta(uv) - u\Delta v - v\Delta u + \frac{4(n-2)}{3}Juv .
 \end{equation}
 If $n\geq3$, then
 \begin{equation}
  \label{eqn:eval_ovsienko_redou2}
  \begin{split}
   D_4(u,v) & = \Delta^2(uv) + u\Delta^2v + v\Delta^2 u \\
    & \quad + \frac{2(n-4)}{n+2}\left(\Delta(u\Delta v + v\Delta u) + (\Delta u)(\Delta v)\right) \\
    & \quad - \frac{2(4n^2-17n+22)}{3(n+2)} \bigl( \delta\left(J\,d(uv)\right) + u\delta\left(J\,dv\right) + v\delta\left(J\,du\right) \bigr) \\
    & \quad + \frac{2(n+2)}{3}\bigl( \delta\left(P(\nabla(uv))\right) + u\delta\left(P(\nabla v)\right) + v\delta\left(P(\nabla u)\right) \bigr) \\
    & \quad + \frac{8(n-1)(n-4)}{3(n+2)}Q_4uv + \frac{8(n-4)^3}{9(n+2)}\sigma_2uv .
  \end{split}
 \end{equation}
 In particular, both $D_1$ and $D_2$ are formally self-adjoint.
\end{thm}

\begin{proof}
 By Definition,
 \[ \cD_2(\cu,\cv) = -\cDelta(\cu \cv) - \cu\cDelta\cv - \cv\cDelta\cu \]
 for $\cu,\cv\in\cmE\bigl[-\frac{n-2}{3}\bigr]$.  Let $u,v\in C^\infty(M)$ and set $\cu=t^{-\frac{n-2}{3}}u$ and $\cv=t^{-\frac{n-2}{3}}v$, where $u,v$ are extended to $\cmG$ by requiring that $\partial_tu=\partial_\rho u=0$ and $\partial_tv=\partial_\rho v=0$.  It follows from~\eqref{eqn:Laplacian} that
 \[ \left.\cD(\cu,\cv)\right|_{t=1,\rho=0} = -\Delta(uv) - u\Delta v - v\Delta u + \frac{4(n-2)}{3}Juv . \]
 Equation~\eqref{eqn:eval_ovsienko_redou1} follows from the fact that $\cD_2$ is tangential.
 
 By definition,
 \[ \cD_4(\cu,\cv) = \cDelta^2(\cu\cv) + \cu\cDelta^2\cv + \cv\cDelta^2\cu + \frac{2(n-4)}{n+2}\left[ \cDelta\left(\cu\cDelta\cv + \cv\cDelta\cu\right) + (\cDelta\cu)(\cDelta\cv) \right] \]
 for $\cu,\cv\in\cmE\bigl[-\frac{n-4}{3}\bigr]$.  Let $u,v\in C^\infty(M)$ and set $\cu=t^{-\frac{n-4}{3}}u$ and $\cv=t^{-\frac{n-4}{3}}v$, where $u,v$ are extended to $\cmG$ by requiring that $\partial_tu=\partial_\rho u=0$ and $\partial_tv = \partial_\rho v = 0$.  It follows from~\eqref{eqn:Laplacian} that
 \begin{align*}
  \left.\cDelta^2(\cu\cv)\right|_{t=1,\rho=0} & = \biggl[ \Delta^2 - \frac{n-4}{3}J\Delta - \frac{2(n-4)}{3}\Delta J - \frac{n+2}{3}\delta(Jg-2P)d \\
   & \qquad + \frac{4(n-1)(n-4)}{9}J^2 - \frac{2(n+2)(n-4)}{9}\lv P\rv^2\biggr](uv) , \\
  \left.\cDelta^2\cu\right|_{t=1,\rho=0} & = \biggl[ \Delta^2 - \frac{2(n-4)}{3}J\Delta - \frac{n-4}{3}\Delta J + \frac{n-10}{3}\delta(Jg-2P)d \\
   & \qquad + \frac{(n+2)(n-4)}{9}J^2 + \frac{(n-4)(n-10)}{9}\lv P\rv^2 \biggr] u ,
 \end{align*}
 where the quantities in the square brackets are compositions of operators; e.g.\ $[\Delta J](u):=\Delta(Ju)$.  Applying~\eqref{eqn:Laplacian} again yields
 \begin{align*}
  \left.\cDelta\left(\cu\cDelta\cv\right)\right|_{t=1,\rho=0} & = \Delta(u\Delta v) - \frac{n-4}{3}Ju\Delta v - \frac{n+2}{3}u\delta\left((Jg-2P)(\nabla v)\right) \\
   & \quad + \frac{n-4}{3}\left[ -\Delta J + \frac{2(n-1)}{3}J^2 - \frac{n+2}{3}\lv P\rv^2\right](uv) , \\
  \left.(\cDelta\cu)(\cDelta\cv)\right|_{t=1,\rho=0} & = (\Delta u)(\Delta v) - \frac{n-4}{3}Ju\Delta v - \frac{n-4}{3}Jv\Delta u + \left(\frac{n-4}{3}\right)^2J^2uv .
 \end{align*}
 Combining the previous two displays and using the fact that $\cD_4$ is tangential yields
 \begin{align*}
  D_4(u,v) & = \Delta^2(uv) + u\Delta^2v + v\Delta^2 u + \frac{2(n-4)}{n+2}\left(\Delta(u\Delta v + v\Delta u) + (\Delta u)(\Delta v)\right) \\
   & \quad - \frac{2(n-2)(n-4)}{n+2}\Bigl(\Delta(Juv) + Ju\Delta v + Jv\Delta u\Bigr) \\
   & \quad - \frac{n-4}{3}\Bigl( (u\Delta(Jv) + v\Delta(Ju) + J\Delta(uv)\Bigr) \\
   & \quad - \frac{n+2}{3}\Bigl(\delta(Jg-2P)d(uv) + u\delta(Jg-2P)dv + v\delta(Jg-2P)du \Bigr) \\
   & \quad + \frac{4(n-4)(4n^2-11n+16)}{9(n+2)}J^2uv - \frac{4(n+2)(n-4)}{9}\lv P\rv^2uv .
 \end{align*}
 Rewriting this using the identities
 \begin{align*}
  \Delta(Juv) + Ju\Delta v + Jv\Delta u - uv\Delta J & = \delta\left(J\,d(uv)\right) + u\delta\left(J\,dv\right) + v\delta\left(J\,du\right) , \\
  u\Delta(Jv) + v\Delta(Ju) + J\Delta(uv) - 2uv\Delta J & = \delta\left(J\,d(uv)\right) + u\delta\left(J\,dv\right) + v\delta\left(J\,du\right)
 \end{align*}
 yields~\eqref{eqn:eval_ovsienko_redou2}.
\end{proof}

%% file: comments.tex
\section{A family of critical-order operators}
\label{sec:comments}

We now give our final ambient obstruction of formally self-adjoint conformally covariant polydifferential operators.  First, we construct a family of $\ell$-differential operators, $\ell\equiv 3\mod 4$, of all critical orders.  Second, we construct a family of $j$-differential operators, $j\equiv 5\mod 6$, of all critical orders, assuming Conjecture~\ref{conj:ovsienko_redou} holds.

In terms of the ambient metric, these operators are obtained by considering certain combinations of the divergence, powers of the Laplacian, and the inner product of gradients.  This task is greatly simplified by the following proposition.

\begin{prop}
 \label{prop:reduction}
 Let $(M^n,g)$ be a Riemannian manifold and suppose that
 \[ \cD \colon \bigl(\cmE\left[-2\right]\bigr)^{\ell} \to \cmE\left[2-n\right] \]
 is a tangential operator in the ambient space which induces a formally self-adjoint operator on $(M^n,g)$.  Define
 \[ \cF \colon \bigl(\cmE\left[0\right]\bigr)^{2\ell+1} \to \cmE\left[-n\right] \]
 by
 \[ \cF(u_1,\dotsc,u_\ell,v_1,\dotsc,v_\ell,w) := \cdelta\left( \cD(\lp\cnabla u_1,\cnabla v_1\rp,\dotsc,\lp\cnabla u_\ell,\cnabla v_\ell\rp)\,\cd w\right) . \]
 Then $\cF$ is tangential and $\Sym\cF$ induces a formally self-adjoint conformally covariant polydifferential operator on $(M^n,g)$.
\end{prop}

\begin{proof}
 Since $\cD$ and the inner product $(u,v)\mapsto\lp\cnabla u,\cnabla v\rp$ are symmetric, it suffices to prove that
 \begin{align}
  \label{eqn:tangential1} \cF(u_1,\dotsc,u_\ell,v_1,\dotsc,v_\ell,Qz) & \equiv 0 \mod Q , \\
  \label{eqn:tangential2} \cF(Qz,u_2,\dotsc,u_\ell,v_1,\dotsc,v_\ell,w) & \equiv 0 \mod Q
 \end{align}
 for any $u_1,\dotsc,u_\ell,v_1,\dotsc,v_\ell,w\in\cmE[0]$ and any $z\in\cmE[-2]$.  On the one hand,
 \begin{align*}
  \cF(u_1,\dotsc,u_\ell,v_1,\dotsc,v_\ell,Qz) & \equiv 2(n+2)\cf w + 2X(\cf z) + 2\cf Xz \mod Q
 \end{align*}
 for $\cf:=\cD(\lp\cnabla u_1,\cnabla v_1\rp,\dotsc,\lp\cnabla u_\ell,\cnabla v_\ell\rp)$.  The facts $Xz=-2z$ and $X(\cf z)=-n\cf z$ imply that~\eqref{eqn:tangential1} holds.  On the other hand, since $v_1\in\cmE[0]$, it holds that $\lp\cnabla (Qz),\cnabla v_1\rp=Q\lp\cnabla z,\cnabla v_1\rp$.  As $\cD$ is tangential on $\bigl(\cmE[-2]\bigr)^\ell$, we conclude that there is a $\cphi\in C^\infty(\cmG)$ such that
 \[ \cD(\lp\cnabla (Qz),\cnabla v_1\rp,\dotsc,\lp\cnabla u_\ell,\cnabla v_\ell\rp) = Q\cphi . \]
 Combining this with the identity
 \[ \cdelta(Q\cphi\,\cd w) \equiv 2\cphi\,Xw \mod Q \]
 yields~\eqref{eqn:tangential2}.  Therefore $\cF$, and hence $\Sym\cF$, is tangential.
 
 We now show that $\Sym\cF$ induces a formally self-adjoint operator on $(M,g)$.  It is straightforward to check using~\eqref{eqn:ambient_metric} that if $\cu:=t^\mu u\in\cmE[0]$ and $\cphi:=t^\nu\phi\in\cmE[2-n]$, where $u=u(x,\rho)$ and $\phi=\phi(x,\rho)$, then
 \begin{equation}
  \label{eqn:tangential_divergence}
  \left.\cdelta\left(\cphi\,\cd\cu\right)\right|_{\rho=0,t=1} = \delta\left(\phi\,du\right) .
 \end{equation}
 Let $D\colon C^\infty(M)^{\ell}\to C^\infty(M)$ be the operator induced by $\cD$.  We conclude from~\eqref{eqn:tangential_divergence} that the operator induced by $\Sym\cF$ is
 \begin{multline*}
  F(u_1,\dotsc,u_{2\ell+1}) = \frac{1}{(2\ell+1)!} \\ \times \sum_{\sigma\in S_{2\ell+1}} \delta\left( D\left(\lp\nabla u_{\sigma(1)},\nabla u_{\sigma(2)}\rp,\dotsc,\lp\nabla u_{\sigma(2\ell-1)},\nabla u_{\sigma(2\ell)}\rp\right)\,du_{\sigma(2\ell+1)} \right) .
 \end{multline*}
 It follows that
 \begin{multline*}
  -(2\ell+1)!\int_M u_0\,F(u_1,\dotsc,u_{2\ell+1}) \\ = \sum_{\sigma\in S_{2\ell+1}} \int_M \lp\nabla u_0,\nabla u_{\sigma(1)}\rp\,D\left(\lp\nabla u_{\sigma(2)},\nabla u_{\sigma(3)}\rp,\dotsc,\lp\nabla u_{\sigma(2\ell)},\nabla u_{\sigma(2\ell+1)}\rp \right) .
 \end{multline*}
 Since $D$ is formally self-adjoint, we conclude that
 \begin{multline*}
  -(2\ell+2)!\int_M u_0\,F(u_1,\dotsc,u_{2\ell+1}) \\ = \sum_{\sigma\in S_{2\ell+2}} \int_M \lp\nabla u_{\sigma(0)},\nabla u_{\sigma(1)}\rp\,D\left(\lp\nabla u_{\sigma(2)},\nabla u_{\sigma(3)}\rp,\dotsc,\lp\nabla u_{\sigma(2\ell)},\nabla u_{\sigma(2\ell+1)}\rp \right) .
 \end{multline*}
 In particular, $F$ is formally self-adjoint.
\end{proof}

We present two applications of Proposition~\ref{prop:reduction}.  The first uses the GJMS operators to construct a suitable operator $\cD$.  This yields the aforementioned family of $j$-differential operators, $j\equiv3\mod 4$.

\begin{cor}
 \label{cor:rank0mod4}
 Let $(M^{2k},g)$ be an even-dimensional Riemannian manifold and let $\ell\in\bN$ be such that $2\ell\leq k$.  Define $\cD\colon\bigl(\cmE[-2]\bigr)^{2\ell-1}\to\cmE[2-2k]$ as the symmetrization of
 \[ (u_1,\dotsc,u_{2\ell-1}) \mapsto u_1\dotsm u_{\ell-1}\cDelta^{k-2\ell}(u_\ell\dotsm u_{2\ell-1}) \]
 and define $\cF\colon\bigl(\cmE[0]\bigr)^{4\ell-1}\to\cmE[-2k]$ by
 \[ \cF(u_1,\dotsc,u_{4\ell-1}) := \cdelta\left(\cD\left(\lp\cnabla u_1,\cnabla u_2\rp,\dotsc,\lp\cnabla u_{4\ell-3},\cnabla u_{4\ell-2}\rp\right)\,\cd u_{4\ell-1}\right) . \]
 Then $\cF$ is tangential and $\Sym\cF$ induces a formally self-adjoint conformally covariant polydifferential operator on $(M,g)$.
\end{cor}

\begin{proof}
 In dimension $2k$, the operator $\cDelta^{k-2\ell}\colon\cmE\left[-2\ell\right]\to\cmE\left[2\ell-2k\right]$ is tangential~\cite{GJMS1992} and induces~\cite{FeffermanGraham2013,GrahamZworski2003} a formally self-adjoint operator on $(M,g)$, namely the GJMS operator of order $2k-4\ell$.  It follows that $\cD\colon\bigl(\cmE[-2]\bigr)^{2\ell-1}\to\cmE[2-2k]$ is tangential and induces a formally self-adjoint operator on $(M,g)$.  The final conclusion follows from Proposition~\ref{prop:reduction}.
\end{proof}

\begin{remark}
 We have restricted to the case of critical dimension --- that is, to finding a conformally covariant polydifferential operator $F$ acting on $\mE[0]$ --- in order to indicate the variety of constructions of polydifferential operators while keeping relatively simple formulas.  We expect that Proposition~\ref{prop:reduction} can be extended to general dimensions.  For example, it is straightforward to show that if $(\cmG^{n+2},\cg)$ is the ambient space of an $n$-dimensional Riemannian manifold, then
 \begin{multline}
  \label{eqn:operatorBk}
  \cF(u,v,w) := \cdelta\left(\cDelta^{k-2}(\lp\cnabla u,\cnabla v\rp)\,\cd w\right) \\ - \frac{n-2k}{8(k-1)}\left[ w\cDelta^{k-1}(\lp\cnabla u,\cnabla v\rp) - \cdelta\left(\cDelta^{k-1}(uv)\,\cd w\right) \right]
 \end{multline}
 is tangential on $\cmE\bigl(\bigl[-\frac{n-2k}{4}\bigr]\bigr)^{3}$ for all integers $2\leq k\leq n/2$ (cf.\ Proposition~\ref{prop:operatorB}).  When $n=2k$, this recovers the rank four example of Corollary~\ref{cor:rank0mod4}.  However, generalizing~\eqref{eqn:operatorBk} to higher ranks requires adding additional terms with coefficients factoring through $(n-2k)^j$, $j\geq2$.
\end{remark}

Our second application uses the Ovsienko--Redou operators to satisfy the hypothesis of Proposition~\ref{prop:reduction}.  This yields the aforementioned family of $j$-differential operators, $j\equiv 5\mod 6$.

\begin{cor}
 \label{cor:rank0mod6}
 Let $(M^{2k},g)$ be an even-dimensional Riemannian manifold and let $\ell\in\bN$ be such that $3\ell\leq k$.  Define $\cD\colon\bigl(\cmE[-2]\bigr)^{3\ell-1}\to\cmE[2-2k]$ as the symmetrization of
 \[ (u_1,\dotsc,u_{3\ell-1}) \mapsto u_1\dotsm u_{\ell-1}\cD_{2k-6\ell}(u_\ell\dotsm u_{2\ell-1},u_{2\ell}\dotsm u_{3\ell-1}) , \]
 where $\cD_{2k-6\ell}$ is the operator of Theorem~\ref{thm:ovsienko_redou}.  Define $\cF\colon\bigl(\cmE[0]\bigr)^{6\ell-1}\to\cmE[-2k]$ by
 \[ \cF(u_1,\dotsc,u_{6\ell-1}) := \cdelta\left(\cD\left(\lp\cnabla u_1,\cnabla u_2\rp,\dotsc,\lp\cnabla u_{6\ell-3},\cnabla u_{6\ell-2}\rp\right)\,\cd u_{6\ell-1}\right) . \]
 Then $\cF$ is tangential and $\Sym\cF$ induces a conformally covariant polydifferential operator on $(M,g)$.  Moreover, if either $\ell=\lfloor k/3\rfloor$ or Conjecture~\ref{conj:ovsienko_redou} holds, then the operator induced by $\Sym\cF$ is formally self-adjoint.
\end{cor}

\begin{proof}
 In dimension $2k$, Theorem~\ref{thm:ovsienko_redou} implies that $\cD_{2k-6\ell}\colon\cmE[-2\ell]\to\cmE[2\ell-2k]$ is tangential.  We conclude from the proof of Proposition~\ref{prop:reduction} that both $\cF$ and $\Sym\cF$ are tangential.
 
 If $\ell=\lfloor k/3\rfloor$, then $k-3\ell<3$, and hence Theorem~\ref{thm:eval_ovsienko_redou} implies that the operator induced by $\cD$ is formally self-adjoint.  If instead Conjecture~\ref{conj:ovsienko_redou} holds, then the operator induced by $\cD$ is formally self-adjoint.  Thus in either case the operator induced by $\cD$ is formally self-adjoint, and hence Proposition~\ref{prop:reduction} also implies that the operator induced by $\Sym\cF$ is formally self-adjoint.
\end{proof}